\documentclass[twoside,11pt]{article}

\usepackage{amsfonts,amsmath,amssymb,enumerate}
\usepackage{euscript,mathrsfs}
\usepackage[dvips]{graphicx}
\usepackage{psfrag,graphicx}
\usepackage{amsthm}
\usepackage[overload]{empheq}
\usepackage{dsfont}

\usepackage{color,graphics}
\usepackage{graphics}
\usepackage{epsfig}
\usepackage{subfigure}
\usepackage{soul}

\newtheorem{theorem}{Theorem}[section]
\newtheorem{lemma}[theorem]{Lemma}
\newtheorem{corollary}[theorem]{Corollary}

\newtheorem{remark}[theorem]{Remark}

\numberwithin{equation}{section}

\def\Vo{\vbox{\offinterlineskip\hbox{\kern 3pt$\scriptstyle\circ$}
\kern 1pt\hbox{$V$}}}

\def\Ho{\vbox{\offinterlineskip\hbox{\kern 3pt$\scriptstyle\circ$}
\kern 1pt\hbox{$H$}}}

\def\Wo{\vbox{\offinterlineskip\hbox{\kern 3pt$\scriptstyle\circ$}
\kern 1pt\hbox{$W$}}}

\newcommand{\inj}{\hookrightarrow}

\newcommand{\bt}{\begin{theorem}}
\newcommand{\et}{\end{theorem}}

\def\eqldef{\overset{\text{\tiny \rm def}}{=}}
\def\tra{\mathsf{T}}

\newcommand{\abs}[2][{}]{\lvert#2\rvert_{#1}}
\def\dd{\;\!\mathrm{d}}    % d for integrals 

\newcommand{\Er}{\mathbb{R}}
\newcommand{\En}{\mathbb{N}}

\newcommand{\norm}[2][{}]{\lVert#2\rVert_{{#1}}}

\newcommand{\bH}{\mathbf{H}}
\newcommand{\bL}{\mathbf{L}}

\newcommand{\bfL}{\mathbf{L}}

\def\tra{\mathsf T}
\def\eqldef{\overset{\text{\tiny\rm def}}{=}}
\def\L{{\rm{L}}}
\def\H{{\rm{H}}}

\def\E{{\rm{E}}}
\def\W{{\rm{W}}}

\def\C{{\mathrm{C}}}
\def\H{{\mathrm{H}}}
\def\L{{\mathrm{L}}}

\def\W{{\mathrm{W}}}

\def\bproof{\begin{proof}}
\def\eproof{\end{proof}}

\newcommand{\bfH}{\mathbf{H}}

\newcommand{\bFormula}[1]{
\begin{equation} \label{#1}}
\newcommand{\eF}{\end{equation}}

\newcommand{\Bd}{\partial \Omega}

\def\longrightharpoonup{\relbar\joinrel\rightharpoonup}

\begin{document}

\title{Solvability for a drift-diffusion
system\\ with Robin boundary
conditions}

%\title{Well-posedness of a drift-diffusion
%system\\ endowed with Robin boundary
%conditions}

\author{A.~Heibig\thanks{Universit\'e de Lyon, CNRS, INSA de Lyon
Institut Camille Jordan UMR 5208, 20 Avenue A. Einstein, F--69621 Villeurbanne, France 
(\tt arnaud.heibig@insa-lyon.fr, apetrov@math.univ-lyon1.fr, christian.reichert@insa-lyon.fr)}
\and A.~Petrov$^{*}$ \and C.~Reichert$^{*}$}

\pagestyle{myheadings}
\thispagestyle{plain}
\markboth{A.~Heibig, A.~Petrov, C. Reichert}
{\small{The drift-diffusion
system with Robin boundary
conditions}}

\maketitle

\begin{abstract}
\noindent This paper focuses on a drift-diffusion system subjected to boundedly non dissipative 
Robin boundary conditions.
A general existence 
result with large initial conditions is established by using suitable \(\L^1\),
\(\L^2\) and trace estimates.
Finally, two examples coming from the corrosion and the self-gravitation model are
 analyzed.
\end{abstract}

\hspace*{-0.6cm}{\textbf {Key words.}}
Drift-diffusion system, Robin boundary conditions, 
complex interpolation, corrosion model, self-gravitation model.
%\begin{keywords}
%\end{keywords}
\vspace{0.3cm}

\hspace*{-0.6cm}{\textbf {AMS Subject Classification.}}
35B45, 35A65, 35D30, 76R50

%\begin{DOI}
%\end{DOI}

\section{Introduction}
\label{Intro}
The drift-diffusion equations have a vast
phenomenology and are currently studied.
When coupled with fluid flows equations, the
resulting systems are usually quite 
complex due to the micro-macro
effect. A short reference list is 
\cite{CiHePa1, CiHePa2, ConstantinMasmoudi, ConSer10GRNS, Jourdain, Lin}
for some problems arising in different contex, 
including the theory of dilute or melt 
polymers. Apart from the theory of stochastic
process -- mainly the Fokker-Planck equation --
a priviledged field of application is the
theory of semi conductors. This includes
systems of Debye type studied for instance 
in \cite{BilDol00LBND, BiHeNa94DSEL, BilNad02GEGS, FanIto95GSDD, FanIto95TDDD, 
Gaj, JunPen00HHMP}.
Let's also mention, 
in the area of chemotaxis, the Patlak-Keller-Segel system
(see  \cite{BMC, KellerSegel, Mizoguchi, Patlak},
and references therein).

In this paper, we focus on the
following problem
\begin{subequations}
\label{intro1_1}
\begin{align}
\label{intro1}
&\partial_tu=\nabla\cdot(\nabla u+
u_{\alpha} \otimes \nabla\mathcal{V}),\quad (t,x) \in(0,T)\times\Omega,\\&
\label{intro2}
u(0) = u_0,
\end{align}
\end{subequations}
with \(\Omega \subset \mathbb{R}^d\)
a smooth bounded domain, 
\(u = (u^1,\ldots, u^n)\) and  
\(u_{\alpha} = (\alpha^1u^1,\ldots, 
\alpha^nu^n)\) \((\alpha^i\in
\mathbb{R} \textrm { for } i =1,\ldots, n)\).
The potential 
\(\mathcal{V}\)
is given by 
\(\mathcal{V}(t) = \mathcal{B}(t, u(t))\)
for a.e \(t\in (0, T)\), with
\(\mathcal{B}: \mathbb{R}_+\times
{\bfL}^1(\Omega)\rightarrow 
\W^{1, \infty}(\Omega)\cap\W^{2, 1}(\Omega)\)
a suitable smoothing,
nonlinear operator.
The boundary conditions on the 
fluxes are of Robin type, which reads as follows
\begin{equation}
\label{intro3}
\Bigl(\frac{\partial u^i}{\partial \nu}+\alpha^i u^i
\frac{\partial\mathcal{V}}{\partial \nu}\Bigr)(t, x)=\sigma^i(t, x,u^i_{|_{\partial\Omega}}(t, x), 
\mathcal{V}_{|_{\partial\Omega}}(t, x)),\quad (t, x)\in[0, T]\times\partial\Omega.
\end{equation}
The fluxes \(\sigma^i\)
are endowed with boundedly
non dissipative conditions,
reminiscent of Kru\v{z}kov entropy conditions:
for all \((t,x, v, \psi)\in [0,T]\times\Bd\times\Er\times\Er\)
\begin{subequations}
\label{intro4_1}
\begin{align}
&\label{intro4}
\sigma^i(t,x,v,\psi)\chi^+(v-k^i)\leq \Lambda_T,\\&
\label{intro5}
\sigma^i(t,x,v,\psi)\chi^-(v)\leq 0,
\end{align}
\end{subequations}
where \(\chi^+\) is the Heaviside 
function and \(\chi^-(v) = -\chi^+(-v)\) and \(k^i>0\).
The goal of the paper is to prove well
posedeness of such a system in a \(\L^2\)
frame (see Theorem \ref{Thetheorem}). 
Let's mention  the close connexion of 
the above equations and the theory of the Navier-Stokes
equations as developed by \cite{ Kato, Weissler} (see also \cite{LR}). 
Nevertheless, we will not use 
this closeness in the present paper, but rather some features of the
$\L^1$ theory of Kru\v{z}kov
for scalar conservation laws. 
See \cite{Kruzkov} and compare with assumptions
\eqref{intro4_1}.

The above problem is  a compromise between
realistic equations such as the Debye 
system, and a more abstract setting.  
Notice that the usual  \(2\times 2\)  semi-conductor model 
(see \cite{BiHeNa94DSEL}) corresponds to the resolvent of the 
Poisson-Dirichlet problem, i.e 
\(\mathcal{B}(t, \cdot) = \Delta_D^{-1}\). Such a resolvent
has relatively bad smoothing properties in
a \(\L^{\infty}\) frame. But, as a  
compensation,
the system admits opposite sign on the nonlinearities 
ensuring large data global solutions.
Contrarily, this sign condition is not fulfilled for the present
system \eqref{intro1_1}, \eqref{intro3}
and we assume the above 
smoothing assumption on the operator
\(\mathcal{B}\). This assumption prevents
us  to apply our results
to the case 
\(\mathcal{B} = \Delta^{-1}_D\)
for \(d\geq 2\).
Nevertheless, due to the
special properties of the 1-D  Laplace operator, 
our results apply to the 
one dimensional Debye type system considered in 
\cite{CHAVI}, a problem we had 
primarily in view
 (see also \cite{BaBoC*12NMSC}). 
In that case, 
 our existence result
improves the former result 
in \cite{CHAVI}.
Actually, since we work in a 
\(\L^2\) frame
and remove the 
sign condition of the Debye \(2\times 2\)
system, 
we obtain
an
existence result for a general
\(n\times n\) system \((d=1)\).
 We also remove the restrictive
 conditions on the initial data in 
 \cite{CHAVI}. Finally, to conclude this
 section, note that 
 in the
 case \(\mathcal{B} = \Delta^{-1}_D\),
 \(d\geq 2\), a mollifying process
 can be used on \(\mathcal{B} \) in order
 to recover some classical results of the
 theory. We treat the simple case of the
 self gravitational system at the end of 
 the paper (Section \ref{zamples}). 

Compared with former works on
the subject (see \cite{BiHeNa94DSEL}),
the novelty of  \cite{CHAVI}
and of the present paper
lies in the Robin boundary condition
\eqref{intro3}. The issue, when dealing
with such a non dissipative condition
is to derive an
\(\L^{\infty}(0, T, \L^1(\Omega))\) estimate 
on the function \(u\) since no decrease or conservation
of \(\norm[\L^1(\Omega)]{u(t)}\) can 
be expected. Thus, the main task is to define and evaluate
the nonlinear term
\(\sigma^i(u^i_{|_{\partial\Omega}})\). 
When working in the classical setting
\(u\in  \L^2(0,t_0;\bH^1(\Omega))
\cap\C^0([0,t_0];\bL^2(\Omega))\), a simple
interpolation procedure shows that the natural
trace space for \(u\) is 
\(\L^q(0,T, \L^q(\Bd))\)
with \(1\leq q < 2+\frac2d\).
This corresponds to a 
restricted class of admissible fluxes, essentially
defined as follows. For
any \((t,x, v, \psi)\in [0,T]\times\Bd\times\Er\times\Er\) with
\(0\leq t \leq T\) and \(\abs{\psi}\leq M\)
\begin{equation}
\label{NLF}
\abs{\sigma^i(t,x,v,\psi)-\sigma^i(t,x,\bar{v},\bar{\psi})}
\leq C_{T, M}
((1{+}\abs{v}^{\rho}{+}\abs{\bar{v}}^{\rho})
\abs{v-\bar{v}}+
(1{+}\abs{v}^{\rho+1}{+}
\abs{\bar{v}}^{\rho+1})
\abs{\psi-\bar{\psi}}),
\end{equation}
with \(0\leq\rho <1+\frac2d\). Within such a class of fluxes, the 
classical existence results of \cite{Lady}
do not apply to the 
natural linearized versions of the system 
\eqref{intro1}, at least for
\(\rho\) close to \(1+\frac2d\). 
As a matter of fact, such 
fluxes leads to rather discontinuous
right
hand sides in the variational
formulations, so that getting  an existence result require an
indirect procedure and the use of all the conditions \eqref{intro4_1} and \eqref{NLF} on the flux.

The paper is organized 
as follows. The equations are described 
in Subsection \ref{Model}
 while a first
simplified set of 
constitutive assumptions is described
in  Subsection
\ref{Math_assump}. 
Essentially, we 
replace condition 
\eqref{NLF} by a global Lipschitz
condition, in order to get a tractable
proof of the existence result 
given in Section 
\ref{lasection}.
The proof of this existence result relies on the aforementioned
\(\L^{\infty}(0, T, \L^1(\Omega))\) estimate.
Since the extensions we have in view 
(Section \ref{lasection})
require some
uniform estimates, we keep
track of the constants (Lemma
 \ref{GE:Lemma3}). Some trace inequalites are 
 established in Section
 \ref{Estimates_trace}, leading to
 the definition of an extended set
 of assumptions. Under these conditions,
 a general existence theorem 
 with large initial data is established in
 Section \ref{lasection} by using some
 ad-hoc density argument.
 The final Section 
 \ref{zamples} is devoted to two 
 realistic examples. The first one deals with a drift-diffusion system with 
 Robin boundary conditions with an application to a corrosion model 
 (cf.~\cite{BaBoC*12NMSC,CHAVI}), while the second one
 is the classical equation of self-gravitation
 system studied for instance in
\cite{BilNad02GEGS}.

\section{Mathematical formulation}
\label{Math_form}

\subsection{The model}
\label{Model}
Let \(T>0\), let \(\Omega\) be a
smooth  bounded
subset of \(\mathbb{R}^d\), and let \(\alpha^i\in\Er\)
be some real given numbers
 (\(d\in 
\mathbb{N}^*, i\in \{1,\ldots, n\}\)). Let 
\(u^i(t,x)\) and \(\mathcal{V}(t)\) be scalar valued functions 
depending on time \(t\).  
Set \(u\eqldef(u^1,\ldots, u^n)\) and denote by
\(\partial_j\) the partial derivative
with respect to the \(j^{\text{th}}\) spatial
variable (\(i\in \{1,\ldots, n\}\)). The mathematical problem is formulated as follows:
\begin{subequations}
\label{eq:1}
\begin{align}
\label{eq:11}
&\forall i\in\{1,\ldots,n\}:\;
\partial_tu^i-\sum_{j=1}^d\partial_j
\big(\partial_j u^i
+\alpha^i u^i \partial_j\mathcal{V}\big)
=0,\quad (t,x) \in(0,T)\times\Omega,\\
\label{eq:12}
&
\mathcal{V}(t) = \mathcal{B}(t, u(t)) 
\quad \textrm{ for a.e }\quad t \in(0,T).
\end{align}
\end{subequations}
The operator \(\mathcal{B}\)
as well as the fluxes \(\sigma^i\) in equation 
\eqref{eq:13_1} below will be precised in the next subsection. 

We now turn to define the boundary
conditions. In the sequel
we
denote by \(\frac{\partial}{\partial \nu}\) the derivative with respect to the
outward normal to \(\Bd\).
The trace of \(u(t)\) on \(\Bd\)
is denoted
by \(u_{|_{\partial\Omega}}(t)\), or more 
often and abusively, by \(u(t)\).
The Robin boundary conditions for \(u^i\) 
are prescribed by
\begin{equation}
\label{eq:13_1}
\forall i\in\{1,\ldots,n\}:\,
\Bigl{(}\frac{\partial u^i}{\partial \nu}+\alpha^i u^i
\frac{\partial\mathcal{V}}{\partial \nu}\Bigr{)}(t, x)=
\sigma^i(t, x,u^i_{|_{\partial\Omega}}, 
\mathcal{V}_{|_{\partial\Omega}}),\quad (t, x)
\in[0, T]\times\partial\Omega,
\end{equation}
and our problem is completed by the following initial conditions:
\begin{equation}
\label{eq:14}
\forall i\in\{1,\ldots,n\}:\; u^i(0)=u_0^i.
\end{equation}
In the sequel, we will often use notations
such as \(u\eqldef(u^1,\ldots,u^n)\) or \(\sigma(t, x, u(t, x), \mathcal{V}(t,x))
\eqldef(\sigma^1\big(t, x, u^1(t, x), \mathcal{V}(t,x)),\ldots,\sigma^n
\big(t, x, u^n(t, x), \mathcal{V}(t,x)))\)
without any comments. We will also write \(u_{\alpha}\) in place of the column
vector with components
\(\alpha^i u^i\)
and \(\nabla\mathcal{V}\) in place of the 
row vector with components \(\partial_j\mathcal{V}\) 
\((i\in\{1,\ldots,n\})\). 
With this last notation,  
equations \eqref{eq:11} can be written in 
the more compact form
\begin{equation}
\label{eq:011}
\partial_tu=\nabla\cdot(
u_{\alpha} \otimes \nabla\mathcal{V}+\nabla u),\quad (t,x) \in(0,T)\times\Omega.
\end{equation}
Thus, we introduce the following notation, used throughout this paper: 
if \(X\) is a space of scalar functions, the bold-face notation 
\(\bf{X}\) denotes the space \(X^n\). Define the following sets:
\begin{equation*}
\forall t\in (0,T]:\; \mathcal{Q}_t\eqldef (0,t)\times\Omega
\quad\text{and}\quad\Gamma_t\eqldef (0,t)\times\partial\Omega.
\end{equation*}
In this paper, equation \eqref{eq:011} will 
often be considered in the following variational sense.
Let \(T>0\) and let \(u_0\in\bfL^2(\Omega)\).
Then for \(t_0\in]0,T]\),  the problem \((\mathcal{P}_{t_0})\) is
\begin{equation*}
(\mathcal{P}_{t_0})\hspace{2em}
\begin{cases}
\text{Find }u\in\L^2(0,t_0;{\bf{H}}^1(\Omega))
\cap\C^0([0,t_0];{\bf{L}}^2(\Omega)) 
\text{ with }
\frac{\mathrm{d}u}{\mathrm{dt}}\in\L^2(0,t_0;({\bf{H}}^1(\Omega))')\\
\text{ such that } u(0) = u_0 \text{ and for any } 
w\in \L^2(0, t_0,{\bf{H}}^1(\Omega)):\\ 
\displaystyle{ 
\int_0^{t_0}\Bigl{\langle}\frac{\dd u}{\dd t}(\tau),
w(\tau)\Bigr{\rangle}\dd\tau
+\int_{{\mathcal{Q}}_{t_0}} (\nabla u+
u_{\alpha}\otimes\nabla\mathcal{V})
(\tau, x):\nabla w(\tau, x)\dd x\dd\tau}
\\= 
\displaystyle{\int_{\Gamma_{t_0}}
\sigma(\tau, x, u(\tau, x),
\mathcal{V}(\tau, x))\cdot w(\tau, x)
\dd\mu \dd\tau,}
\end{cases}
\end{equation*}
with \(\mathcal{V}(t)\eqldef 
\mathcal{B}(t, u(t))\) for a.e 
\(t\in(0, t_0)\).\\

In the above, and
throughout this paper, notations 
such as 
\((\nabla v+v_{\alpha}\otimes\nabla\gamma):\nabla w\) stands for 
\(\sum_{i,j}(\partial_jv^i+
\alpha^iv^i\partial_j\gamma)
\partial_jw^i\), and the dot usually denotes
the canonical scalar product in \(\mathbb{R}^n\).
Notation \((\bH^1(\Omega))'\) 
denotes the topological dual of
\({\bf{H}}^1(\Omega)\). We always abridge
the notation
\(\langle\cdot, \cdot \rangle
_{(\bH^1(\Omega))', \bH^1(\Omega)}\)
in 
\(\langle\cdot, \cdot \rangle\).
Last, notation \(\mu\) or \(\dd\mu\) stands for the usual measure
on \(\Bd\). 

%%%%%%%%%%%%%%%%%%%%%%%%%%%%%%%%%%%%%%%%%%%%%%%%%

\subsection{The simplified case}
\label{Math_assump}

In this section, we introduce some assumptions on the constitutive
functions of the problem
and give a few simple consequences of these assumptions. 
The assumption (A--2) 
will be relaxed at the end 
of the paper by using a density argument.  

\begin{enumerate}[(\text{A}--1)]

\item The operator \(\mathcal{B}: \mathbb{R}_+\times
{\bfL}^1(\Omega)\rightarrow 
\W^{1, \infty}(\Omega)\cap\W^{2, 1}(\Omega)\)
is, locally uniformely in \(t\), Lipschitz continuous with respect 
to the second variable, i.e, for any \((v, w)\in 
\bL^1(\Omega)\times \bL^1(\Omega)\), and almost
every \(t\in[0, T]\), we have
\begin{subequations}
\begin{align}
\label{Lip1}
 &
 \norm[\W^{1, \infty}(\Omega)\cap
 \W^{2, 1}(\Omega)]{\mathcal{B}(t,0)}\leq C_T,\\&
\label{Lip2}
 \norm[\W^{1, \infty}(\Omega)\cap
 \W^{2, 1}(\Omega)]{\mathcal{B}(t,v)-\mathcal{B}(t,w)}
 \leq C_T 
 \norm[\L^1(\Omega)]{v-w}.
\end{align}
\end{subequations}

\item The fluxes \(\sigma^i:  [0,\infty)\times\Bd\times \Er\times\Er
\rightarrow\Er\) are measurable, locally bounded functions and satisfy
\begin{equation}
\label{eq:15}
\begin{aligned}
&
\forall M>0,\quad\exists K_M>0: \\
&\forall (t,x)\in[0,M]\times\Bd,\quad
\forall (v,\psi)\in  \Er\times[-M,M], \quad\forall 
(\bar{v},\bar{\psi})\in \Er\times[-M,M]:\\&
\abs{\sigma^i(t,x,v,\psi)-\sigma^i(t,x,\bar{v},\bar{\psi})}
\leq K_M
(
{\abs{v-\bar{v}}}+
\abs{\psi-\bar{\psi}}).
\end{aligned}
\end{equation}

\item The fluxes \(\sigma^i\) are boundedly non dissipative (at height \(k^i\)) in the following sense:
\begin{subequations}
\label{eq:16}
\begin{align}
&\notag\exists \Lambda_T>0,\quad\forall i\in\{1,\ldots,n\},\quad\exists k^i>0:\\&
\label{eq:15_1}
\forall (t,x, v, \psi)\in [0,T]\times\Bd\times\Er\times\Er:\;
\sigma^i(t,x,v,\psi)\chi^+(v-k^i)\leq \Lambda_T,\\&
\label{eq:15_2}
\forall (t,x, v, \psi)\in [0,T]\times\Bd\times\Er\times\Er:
\sigma^i(t,x,v,\psi)\chi^-(v)\leq 0,
\end{align}
\end{subequations}
where \(\chi^+: \Er\rightarrow \Er\) and \(\chi^-:\Er\rightarrow\Er\) are defined by
\begin{equation*}
\chi^+(x)\eqldef
\begin{cases}
1\quad\text{if}\quad x>0\\
0\quad\text{if}\quad x\leq 0
\end{cases}
\text{and}\quad \chi^-(x)\eqldef -\chi^+(-x).
\end{equation*}
\end{enumerate}
Let us make a few comments about these assumptions. 
Notice first that we could replace the assumption (A--1)
on the operator \(\mathcal{B}\) by the following lemma, which is practically 
all what we need in the sequel. In this lemma, and throughout this paper, 
\(\Vert \mathcal{B} \Vert\) denotes the (best) constant 
 $C_T$ in
 \eqref{Lip1} and \eqref{Lip2}.
\begin{lemma}
\label{Lm4} Let \(T>0\). Assume that \({\emph{(A--1)}}\) holds.
Let \(u\) and \(\bar{u}\) belongs
to \(\L^2(0, T;\bH^1(\Omega))\).
Set \(\mathcal{V}(t) \eqldef\mathcal{B}(t, u(t))\)
and 
\(\bar{\mathcal{V}}(t)\eqldef
\mathcal{B}(t, \bar{u}(t))\)
for almost every 
\(t \in (0,T)\). Let \(s\geq 1\).
Then, for a.e.
 \(t\in [0, T]\)
\begin{subequations}
\label{eq:Loc7}
\begin{align}
\label{eq:Loc7_1}
&\norm[\W^{1, \infty}(\Omega)]{\mathcal{V}(t)}
+
\norm[\W^{2, 1}(\Omega)]{\mathcal{V}(t)}
\leq C
 \norm[\bL^{s}(\Omega)]{u(t)}+C ,
\\
\label{eq:Loc7_2}&
\norm[\W^{1, \infty}(\Omega)\cap\W^{2,1}(\Omega)]{(\mathcal{V}-\bar{\mathcal{V}})(t)}
 \leq C
 \norm[\bL^{s}(\Omega)]{(u-\bar{u})(t)},
\end{align}
\end{subequations}
with \(C\eqldef C(T, \Vert \mathcal{B}\Vert)\).
\end{lemma}
Lemma \ref{Lm4} will be mostly used with 
\(s = 2\), but the case \(s= 1\) will
be required when proving a uniform \(\L^{\infty}\)
bound on a family of potential function \(\{\mathcal{V}_p\}_{p\in
\mathbb{N}^*}\). 

Note also that in assumption 
(A--2) we solely demand the local Lipschitz continuity with respect to the \(\psi\) variable, 
in contrast with the global Lipschitz continuity with respect to the \(v\) variable. 
This stems from the fact that in the sequel, the functions \(u^i\)
may not be bounded while we will always 
have
\(\mathcal{V}_{|_{\partial\Omega}}
\in \L^{\infty}(0, T, \L^{\infty}(\Bd))\), due to the regularizing effect
of \(\mathcal{B}\)
(see (A--1)).

In assumption (A--3), we have written the bounded 
non-dissipative conditions at the height \(k^i\). This 
will provide suitable a priori estimates
on \(u^i\) since for  lower
values of \(u^i\), 
we directly have the upper bound
\(u^i \leq k^i\). This upper bound will be completed
by the usual lower bound \(u^i\geq 0\).

%%%%%%%%%%%%%%%%%%%%%%%%%%%%%%%%%%%%%%%%%%%%%%%%%

\subsection{Global existence: the simplified case}
\label{global_exists}

The aim of this subsection consists in showing a global well-posedness for problem
\((\mathcal{P}_T)\) under the simplified set of assumptions (A--1)--(A--3)
(see Corollary~\ref{GE:Cor1}).

The following local existence theorem
can be proved by using a linearized
existence theorem, trace lemmas and the
Picard fixed point theorem. It's proof
is omitted,  since in the simplified case
(i.e under assumption (A--2)), trace 
terms are easy to handle.
\begin{theorem}
\label{Cor6}
Assume that {\emph{(A--1)}} and
{\emph{(A--2)}} hold. Assume that \(u_0\in\bL^2(\Omega)\), and let \(T>0\) be fixed.
Then for \(t_0\in]0,T]\) small enough, the problem
\emph{(\(\mathcal{P}_{t_0}\))}
admits exactly one solution.
Moreover, the time existence
is a nondecreasing function of
\(\norm[\bfL^2(\Omega)]{u_0}\) only.
\end{theorem}
We now proceed with the proof of global
existence. Let  \(\chi^+\) be the Heaviside 
function and \(\chi^-(v) = -\chi^+(-v)\). Let 
\(g\in\C^{\infty}(\Er)\) be an increasing function 
such that 
\begin{equation*}
g(x)\eqldef
\begin{cases}
0\quad\text{if}\quad x\leq0,\\
1\quad\text{if}\quad x\geq 1.
\end{cases}
\end{equation*}
For any \(x\in\mathbb{R}\), let us define \(\chi_{\varepsilon}^+(x)\eqldef g\bigl(
x/\epsilon\bigr)\),
\(\chi_{\varepsilon}^-(x)\eqldef -\chi_{\varepsilon}^+(-x)\)
and 
\( (x)^{\pm}_{\varepsilon}\eqldef
\int_0^x\chi_{\varepsilon}^{\pm}(s)\dd s\).
The following simple lemma collects some
useful properties of these functions.

\begin{lemma}
\label{GE:Lemma1}
\begin{enumerate}[(i)]

\item Let \(U\subset\Er^m\) \((m\in\En^*)\). 
For any \(\varepsilon>0\)
\begin{subequations}
\label{eq:GE2}
\begin{align}
&\label{eq:GE2_4}
0\leq \chi^{\pm}_{\varepsilon}
\chi^{\pm}\leq 1\quad\text{and}\quad
\chi^{\pm}_{\varepsilon}
\chi^{\pm 2}
=\chi^{\pm}_{\varepsilon},\\&
\label{eq:GE2_3}
\forall (f,w)\in\L^1(U)\times\L^1(U):\;
\chi^{\pm}_{\varepsilon}
(w)f
\underset{\varepsilon\rightarrow 0}{\longrightarrow} 
\chi^{\pm}
(w)f\quad\text{in}\quad\L^1(U).
\end{align}
\end{subequations}

\item Let \(T>0\), \(z\in\mathbb{R}\), 
\(j\in\{1,\ldots, n\}\),
\(\varepsilon> 0\),
\(\phi\in\L^2(0, T, \H^1(\Omega))\), 
\(h\in \L^{\infty}((0, T)\times \Omega)\).
Then 
\begin{equation}
\label{eq:GE2_5}
\begin{aligned}
\int_{\Gamma_T}
(\phi - z ) h \partial_j\big(\chi^{\pm}_{\varepsilon}
(\phi-z)\big)\dd x\dd\tau
\underset{\varepsilon\rightarrow 0}{\longrightarrow} 0.
\end{aligned}
\end{equation}
\end{enumerate}
\end{lemma}

\begin{proof}
We only prove \emph{(ii)} for 
\(\chi_{\varepsilon}^{+}\). Let us introduce the following notation:
\(\mathcal{I}_{\varepsilon}\eqldef
\vert
\int_{{\mathcal{Q}}_T} 
(\phi - z ) h \partial_j(\chi_{\varepsilon}^{+}
\circ(\phi-z))\dd x\dd \tau
\vert\). 
Since the support of \(\chi_{\varepsilon}^{+}\) 
is included in \([0, \varepsilon]\), 
and since 
\(\vert \big(\chi_{\varepsilon}^{+}\big)'\vert \leq \frac{C}{\varepsilon}\) we readily obtain 
\begin{equation*}
\mathcal{I}_{\varepsilon}
\leq \int_{[0, T]\times \Omega}
\mathds{1}_{0<\phi-z\leq\varepsilon}
\vert \varepsilon (C/\varepsilon)
h \partial_j
(\phi-z)
 \vert 
\dd x\dd\tau
\end{equation*}
where \( \mathds{1}_{0<\phi-z\leq\varepsilon}\)denotes the indicator function of the
 set \(0<\phi-z\leq\varepsilon\).
By dominated convergence, this last integral
tends to zero with \(\varepsilon\). In fact, 
\(\vert h \partial_j
(\phi-z)\mathds{1}_{0<\phi-z\leq\varepsilon}\vert\leq
\vert h \partial_j
(\phi-z)\vert\in\L^1([0, T]\times \Omega)\)
and 
\(h\partial_j
(\phi-z)\mathds{1}_{0<\phi-z\leq\varepsilon}
\underset{\varepsilon\rightarrow 0}{\longrightarrow} 0\) a.e.
\end{proof}
We now prove some global in time \(\L^1\) and
\(\L^2\) estimates for the solutions
of the problem (\(\mathcal{P}_{T}\)). In the following statement, our main assumptions
are conditions (A--1) and  (A--3). Since the composition operators
have to be well defined, we also assume that
the assumption (A--2) also holds
true. Nevertheless, notice that 
the  estimates of Lemma 
\(\ref{GE:Lemma2}\) do not depend
on the constants \(K_M\) of continuity of the
functions \(\sigma^i\), a fact that will
be used in the next section.
In the sequel, for \(f=(f_1,\ldots,f_N)
\in \bfL^1(U)\), we write 
\(\norm[\bfL^1(U)]{f} = 
\sum_{k=1}^{N}\norm[\L^1(U)]{f_k}\).

\begin{lemma}
\label{GE:Lemma2} 
Assume that {\emph{(A--1)--(A--3)}} hold, and 
assume that \(u_0\in\bfL^2(\Omega)\). Let 
\(T>0\) be given, and let 
\(u\) be any solution to problem 
{\emph{(\(\mathcal{P}_{T}\))}}.
\begin{enumerate}[(i)]

\item Then, for any \(t\in[0,T]\), we have
\begin{subequations}
\begin{align}
\label{eq:GE4_2}
&\norm[\L^1(\Omega)]{(u^i)^-(t)}\leq
\norm[\L^1(\Omega)]{(u^i_0)^-(t)}\\
&\label{eq:GE4_1}
\norm[\L^1(\Omega)]{(u^i-k^i)^+(t)}
\leq
\norm[\L^1(\Omega)]{(u_0^i-k^i)^+}
-\int_{\Gamma_t}
k^i\alpha_i\chi^+(u^i-k^i)\nabla\mathcal{V}\cdot \nu
\dd\mu\dd\tau\\\nonumber&
+\int_{{\mathcal{Q}}_t}  k^i\alpha_i
\chi^+(u^i-k^i)
\Delta\mathcal{V}\dd x\dd\tau+
\mu(\Bd)\Lambda_T
t.
\end{align}
\end{subequations}

\item Assume moreover that
\(u^i\geq 0\), \(i\in\{1,\ldots,n\}\).
Then, we have
\begin{equation}
\label{eq:GE4_3}
\begin{aligned}
&\frac12\norm[\L^2(\Omega)]{u^i(t)}^2
+\norm[\L^2({{\mathcal{Q}}_t})]
{\nabla u^i}^2
\leq
\frac12\norm[\L^2(\Omega)]{u^i_0}^2
-\alpha_i\int_{{\mathcal{Q}}_t} 
u^i
\nabla\mathcal{V}\cdot
\nabla u^i \dd x\dd\tau\\&
+(\Lambda_T + \sup_{A_i(T, 
\norm[\L^{\infty}(0, T, 
\L^{\infty}(\Bd))]{\mathcal{V}})}
\vert \sigma^i \vert )
\int_{\Gamma_t} u^i \dd\mu \dd\tau,
\end{aligned}
\end{equation}
with, for any \(Z\in \mathbb{R}_+\), 
\(A_i(T,  Z)
\eqldef [0, T]\times \Bd \times [0, k^i] \times [-Z,Z]\).
%\label{eq:GE4_4}
\end{enumerate}
\end{lemma}
\begin{proof}
We prove \eqref{eq:GE4_2} and \eqref{eq:GE4_1} at the same time. 
In the sequel, \((\chi_{\varepsilon}^{\pm},(\cdot)_{\varepsilon}^{\pm}, \chi^{\pm},
(\cdot)^{\pm}, z^i)\) denotes either 
\((\chi_{\varepsilon}^+,(\cdot)_{\varepsilon}^{+},\chi^+,(\cdot)^+,
k^i)\) or 
\((\chi_{\varepsilon}^-, (\cdot)_{\varepsilon}^{-}, \chi^-,(\cdot)^-, 0\)). 
Since 
\(\chi_{\varepsilon}^{\pm}(u^i-z^i)\in\L^2(0,T;\H^1(\Omega))\),
 \((\mathcal{P}_{T})\) provides for any \(t\in[0,T]\) 
\begin{equation}
 \label{eq:GE5}
\begin{aligned}
&\int_0^t\Bigl{\langle}\frac{\dd u^i}{\dd t}
(\tau),\chi_{\varepsilon}^{\pm}(u^i-z^i)(\tau)\Bigr{\rangle}\dd\tau\\
&=-\int_{{\mathcal{Q}}_t} 
(\nabla(u^i-z^i)\cdot
\nabla(\chi_{\varepsilon}^{\pm}(u^i-z^i))
+
\alpha_i u^i\nabla\mathcal{V}\cdot
\nabla (\chi_{\varepsilon}^{\pm}(u^i-z^i)))
\dd x\dd\tau\\&
+\int_{\Gamma_t}\sigma^i(\tau,x,u^i(\tau,x),
\mathcal{V}(\tau,x))\chi_{\varepsilon}^{\pm}(u^i(\tau,x)-z^i)\dd\mu\dd\tau.
\end{aligned}
\end{equation}
We estimate the various terms 
appearing in the equality \eqref{eq:GE5}. 
Note first that 
\begin{equation}
\label{1756}
-\int_{{\mathcal{Q}}_t} \nabla(u^i-z^i)\cdot
\nabla(\chi_{\varepsilon}^{\pm}(u^i-z^i))
\dd x\dd\tau \leq 0,
\end{equation}
due to \((\chi_{\varepsilon}^{\pm})'\geq 0\). Next,
since \(u^i\in\L^{\infty}(0,T;\L^2(\Omega))\cap
\L^2(0,T;\H^{1}(\Omega))\), we have
\begin{equation}
\label{1757}
\begin{aligned}
&-
\int_{{\mathcal{Q}}_t} u^i\nabla\mathcal{V}\cdot
\nabla (\chi_{\varepsilon}^{\pm}(u^i-z^i))\dd x\dd\tau 
=
-
\int_{{\mathcal{Q}}_t} (u^i-z^i)\nabla\mathcal{V}\cdot
\nabla (\chi_{\varepsilon}^{\pm}(u^i-z^i))
\dd x\dd\tau\\
&+z^i
\int_{{\mathcal{Q}}_t} \Delta\mathcal{V}
\chi_{\varepsilon}^{\pm}(u^i-z^i)
\dd x\dd\tau
-
z^i
\int_{\Gamma_t}\chi_{\varepsilon}^{\pm}(u^i-z^i)
\nabla\mathcal{V}\cdot
\nu
\dd \mu\dd\tau\\
&\underset{\varepsilon\rightarrow 0}{\longrightarrow} 
z^i
\int_{{\mathcal{Q}}_t} \Delta\mathcal{V}
\chi^{\pm}(u^i-z^i)
\dd x\dd\tau
-
z^i
\int_{\Gamma_t}\chi^{\pm}(u^i-z^i)
\nabla\mathcal{V}\cdot
\nu
\dd \mu\dd\tau,
\end{aligned}
\end{equation}
due to Lemma \ref{GE:Lemma1}, \eqref{eq:GE2_3}
and \eqref{eq:GE2_5}, Lemma 
\ref{Lm4} and a trace lemma. For the boundary term, using 
Lemma \ref{GE:Lemma1}, 
\eqref{eq:GE2_4} and assumption (A--3), we get 
\begin{equation}
\label{1758}
 \begin{aligned}
  &\int_{\Gamma_t}\sigma(\tau,x,u^i(\tau,x),
\mathcal{V}(\tau,x))\chi_{\varepsilon}^{\pm}(u^i(\tau,x)-z^i)
\dd\mu\dd\tau\\
&=
\int_{\Gamma_t}\sigma(\tau,x,u^i(\tau,x),
\mathcal{V}(\tau,x))\chi^{\pm}(u^i(\tau,x)-z^i)
(\chi^{\pm}
\chi_{\varepsilon}^{\pm})(u^i(\tau,x)-z^i)\dd\mu\dd\tau\\
&\leq 
\int_{\Gamma_t}A_T\dd\mu\dd\tau
= \mu(\Bd)tA_T,
 \end{aligned}
 \end{equation}
with 
\begin{equation}
\label{bardeleben}
A_T\eqldef 0 \quad\textrm{for}\quad z^i=0
\quad\textrm{and}\quad A_T\eqldef\Lambda_T
\quad\textrm{for}\quad z^i=k^i.
\end{equation} 
Last, we observe that
for any \(w\in\C^{\infty}([0,T];\H^1(\Omega))\), we have
\begin{equation}
\label{eq:GE6}
\begin{aligned}
\int_0^t\Bigl{\langle}\frac{\dd (w-z^i)}{\dd t}(\tau),\chi_{\varepsilon}^{\pm}(w(\tau)-z^i)
\Bigl{\rangle}\dd\tau
&=
\int_{{\mathcal{Q}}_t} \frac{\dd}{\dd t}(w-z^i)_{\varepsilon}^{\pm}\dd x\dd\tau
\\&=
\int_{\Omega}\bigl((w(t)-z^i)_{\varepsilon}^{\pm}-
(w(0,\cdot)-z^i)_{\varepsilon}^{\pm}
\bigr)\dd x.
\end{aligned}
\end{equation}
Let \(E(u^i)\in \L^2(\mathbb{R}; \H^1(\Omega))
\cap \H^1(\mathbb{R}; \H^1(\Omega)')\)
be an extension of \(u^i\). Denote by
\(\theta \in \mathscr{D}(\mathbb{R})\)
a probability density, and for any 
\(\eta >0, t\in \mathbb{R}\), write 
\(\theta_{\eta}(t) \eqldef \eta^{-1}\theta(\eta^{-1}t)\).
Set in equality \eqref{eq:GE6} 
\(w\eqldef(E(u^i))\star \theta_{\eta}\),
where \(\star \) denotes the convolution with respect to the time
variable.
Letting \(\eta\) tends to \(0\), we see that
\eqref{eq:GE6} holds true with \(w =u^i\).
Furthermore, Lemma \ref{GE:Lemma1} implies that
\begin{equation}
\label{eq:GE7}
\int_0^t\Bigl{\langle}
\frac{\dd (u^i-z^i)}{\dd t}(\tau),\chi_{\varepsilon}^{\pm}(u^i(\tau)-z^i)\Bigl{\rangle}\dd\tau
\underset{\varepsilon\rightarrow 0}{\longrightarrow} 
\int_{\Omega}
((u^i(t)-z^i)^{\pm}-(u_0^i-z^i)^{\pm})\dd x,
\end{equation}
for all  \(i\in\{1,\ldots,n\}\). 
It follows from \eqref{eq:GE5}--\eqref{1758}, and \eqref{eq:GE7} that
\begin{equation}
\label{eq:GE16}
\begin{aligned}
&\int_{\Omega}
(u^i-z^i)^{\pm}\dd x 
-\int_{\Omega}
(u^i_0-z^i)^{\pm} \dd x
\leq 
z^i\alpha_i\int_{{\mathcal{Q}}_T} 
\Delta\mathcal{V}\chi^{\pm}(u^i-z^i)\dd x\dd \tau\\
&-
z^i\alpha_i\int_{{\mathcal{Q}}_T} 
\chi^{\pm}(u^i-z^i)\nabla \mathcal{V}\cdot \nu\dd x\dd\tau
+\mu(\Bd)tA_T.
 \end{aligned}
 \end{equation}
Now, \eqref{eq:GE4_2} and  \eqref{eq:GE4_1}
follow from 
\eqref{bardeleben} and 
\eqref{eq:GE16}.
 Finally, \eqref{eq:GE4_3}
 follows from \((\mathcal{P}_T)\)
 with \(w = (0,\ldots, 0, 
 u^i,0,\ldots, 0)\)
 and the estimate
 \begin{equation*}
\int_{\Gamma_t}
\sigma^i(\tau, x, u^i(\tau, x),
\mathcal{V}(\tau, x))u^i\dd\mu \dd\tau
\leq 
(\Lambda_T + \sup_{A_i(T, 
\norm[\L^{\infty}(0, T, 
\L^{\infty}(\Bd))]{\mathcal{V}})}
\vert \sigma^i \vert )
\int_{\Gamma_t} u^i \dd\mu \dd\tau
 \end{equation*}
 is a consequence of condition (A--3),
 \(u^i\geq 0\) and the definition of 
 \(A_i(T, \mathcal{V})\).
\end{proof}
Let \(t\in [0, T]\). 
We now prove our main \(\L^2\) estimates,
 which hold in the functional spaces
 \(\E_{t}\eqldef \L^2(0,t;\bH^1(\Omega))
\cap\L^{\infty}([0,t];\bL^2(\Omega))\)
For any 
\(v\eqldef(v_1,\ldots, v_n)\in \E_{t}\), set
\begin{equation}
\label{Norme}
\norm[\E_{t}]{v}^2{\eqldef}\sum_{i=1}^n\bigl(
\norm[\L^{\infty}(0,t;\L^2(\Omega))]{v^i}^2+
\norm[\L^2({{\mathcal{Q}}_t})]{\nabla v^i}^2
\bigr).
\end{equation}
Observe that until the end of
 the paper, \(\vert\cdot\vert_1\) denotes
 the \(\ell^1\) norm in \(\mathbb{R}^n\).
 
Before
proceeding, notice the following
inequalities, valid for \(x\geq 0\)
 and \(z\in\mathbb{R}\): 
  \((x-z)^+-\vert z \vert
 \leq x \leq (x-z)^+ +z\). As a consequence,
 for any \(p\in [1, \infty)\), \(z\in \mathbb{R}\),
 \(v\in \L^p(U)\) with \(v\geq 0\)
(and \(U 
\) a bounded domain)
\begin{equation}
\label{Lm44}
 \norm[\L^p(U)]{(v-z)^+}-\vert z\vert 
 \vert U \vert^{1/p}
 \leq 
  \norm[\L^p(U)]{v}
   \leq \norm[\L^p(U)]{(v-z)^+}+\vert z\vert 
 \vert U \vert^{1/p}.
\end{equation}
Together with Lemma \ref{GE:Lemma2}, this provide
 \eqref{GE:Lemma3} below. In the statement of this lemma, 
 we keep track of the dependences with
 respect to the constitutive constants appearing in conditions
 (A--1)--(A--3), since this will turn out to be useful
 in the next section. Nevertheless, 
 we drop in our writings the 
 extraneous dependences such those with respect to \(\Omega\).
%  Last, the (best) constant 
%  $C_T$ in
%  \eqref{Lip1} and \eqref{Lip2} is denoted 
%  by \(\Vert \mathcal{B} \Vert\).
 
\begin{lemma}
 \label{GE:Lemma3}
 Under the assumptions of Lemma \ref{GE:Lemma2}, and assuming
 that
 \(u_0^i\geq 0\) a.e for
 any \(i\in\{1,\ldots,n\}\), we have
\begin{subequations}
 \label{eq:GE17}
 \begin{align}
 \label{eq:GE17_1}
 &u^i(t)\geq 0\text{ for any }t\in[0,T], \text{ x a.e},\\&
\label{eq:GE17_2}
\norm[\L^1(\Omega)]{u(t)}\leq
 C_1
 \mathrm{e}^{C_2t},\\&
 \label{eq:GE17_3}
 \norm[\rm{E}_t]{u}
 \leq
 C_3
 \mathrm{e}^{C_4t}\text{ for all }t\in[0,T],
 \end{align}
\end{subequations}
with 
the notations \(C_1 \eqldef C_1(\norm[\L^1]{u_0},
\Lambda_T, k^1,\ldots, k^n, T,  
\Vert \mathcal{B} \Vert)\), 
\(C_2\eqldef C_2(
\Lambda_T, k^1,\ldots, k^n,  T,  
\Vert \mathcal{B} \Vert)\),
\(C_3 \eqldef C_3(\norm[\L^1]{u_0},
\norm[\L^2]{u_0},
\Lambda_T^{*}, k^1,\ldots, k^n,  T, 
\Vert \mathcal{B} \Vert)\),
\(C_4\eqldef C_4(\norm[\L^1]{u_0},
\Lambda_T^{*}, k^1,\ldots, k^n,  T, 
\Vert \mathcal{B} \Vert)\), \(\Lambda_T^{*}\eqldef\Lambda_T+\sum_{i=1}^n
\sup_
{A_i(T, C_{*})}\vert \sigma^i \vert\) 
and \(C_{*}\eqldef C_{*}(\norm[\L^1]{u_0},
\Lambda_T, k^1,\ldots, k^n, T, 
\Vert \mathcal{B}\Vert)\).
Moreover, the constants \(C_i\) and  \(C_{*}\)
are non-decreasing functions
of their arguments.
\end{lemma} 
\begin{proof}
The Lemma \ref{GE:Lemma2} and assumption \(u_0^i\geq 0\)
a.e imply that \eqref{eq:GE17_1} holds true. 

We now prove inequality \eqref{eq:GE17_2}. 
Before proceeding, remark that 
\eqref{eq:Loc7}
and trace lemmas entail that, for any \(t\in [0, T)\), we have 
\begin{equation}
\label{potential}
\begin{aligned}
\norm[\W^{2, 1}(\Omega)
\cap \W^{1, \infty}(\Omega)]{\mathcal{V}(t)}
+
\norm[\bfL^1(\Bd)]{\nabla\mathcal{V}(t)}
\leq 
C(T, \Vert\mathcal{B}\Vert, \Omega)\big(
\norm[\mathbf{L}^1(\Omega)]{u(t)}+1\big)   
\end{aligned}
\end{equation}
We derive from
inequalities \eqref{eq:GE4_1} and 
\eqref{potential}  that, for any
\(t\in[0, T)\), \(i\in\{1,\ldots,n\}\)
\begin{equation}
\label{2018}
\begin{aligned}
&\norm[\L^1(\Omega)]{(u^i-k^i)^+(t)}
\leq
\Vert (u^i_0-k^i)^+\Vert_{\L^1(\Omega)}
\\&+
\vert
k^i\alpha_i
\vert
C(T, \Vert \mathcal{B}\Vert , \Omega)
(\norm[\L^{1}(0, t, \mathbf{L}^1(\Omega))]{u}
+t)+ \mu(\Bd)\Lambda_Tt.
\end{aligned}
\end{equation}
Taking the sum over
\(i\in\{1,\ldots, n\}\)
and using 
\eqref{Lm44}, we get
\begin{equation}
\label{2028}
\begin{aligned}
 &\norm[\mathbf{L}^1(\Omega)]{u(t)}
 \leq \norm[\mathbf{L}^1(\Omega)]{u_0}
 +\vert k \vert_{1}\vert \alpha \vert_{1}
 C(T, \Vert\mathcal{B}\Vert, \Omega)
(\norm[\L^{1}(0, t, \mathbf{L}^1(\Omega))]{u}
+t)\\&+ n\mu(\Bd)\Lambda_Tt
+
2n\vert
\Omega\vert\vert k \vert_{1}.
\end{aligned}
\end{equation}
Appealing to Gr\"onwall lemma, we obtain 
\eqref{eq:GE17_2}.

We finally
prove inequality \eqref{eq:GE17_3}.
Inequality \eqref{potential} together
with inequality 
\eqref{eq:GE17_2} and a trace lemma, give
\begin{equation}
\label{infini}
\norm[\L^{\infty}(0, T;\W^{1, \infty}(
\Omega))]{\mathcal{V}}+
\norm[\L^{\infty}(\Gamma_T)]{\mathcal{V}}
\leq 
C_{*}
\end{equation}
with \(C_{*}\eqldef C_{*}(\norm[\L^1(\Omega)]{u_0},
\Lambda_T, k^1,\ldots, k^n, T, 
\Vert\mathcal{B}\Vert, \Omega)\).
Recall that for any \(w\in\H^1(\Omega)\), we have
\begin{equation}
\label{Norbert}
\norm[\L^2(\Bd)]{w}
\leq C 
\norm[\H^{3/2}(\Omega)]{w}
\leq
\epsilon \norm[\H^1(\Omega)]{w}+
C_{\epsilon}
\norm[\L^2(\Omega)]{w}
\end{equation}
for any \(\epsilon >0\), so that, integrating
\eqref{Norbert} and using Young inequality, we get
\begin{equation}
\label{naze}
\norm[\L^1(0, t, \L^2(\Bd))]{u^i}
\leq 
t
+ C_{\eta}\norm[\L^2({{\mathcal{Q}}_t} )]{u^i}^2
+\eta
\norm[\L^2(0, t, \bfL^2(\Omega))]{\nabla u^i}^2
\end{equation}
for any \(\eta > 0\). Hence, identity 
\eqref{naze}, \eqref{infini}, \eqref{eq:GE4_3} and Young inequalities
imply that
\begin{equation}
\label{gr}
\begin{aligned}
&\frac12 \norm[\L^2(\Omega)]{u^i(t)}^2
+\norm[\L^2(\mathcal{Q}_t)]
{\nabla u^i}^2\\&
\leq
\frac12 \norm[\L^2(\Omega)]{u^i_0}^2+
C_{*}\abs{\alpha^i}(
C_{\eta}\norm[\L^2({{\mathcal{Q}}_t})]{u^i}^2
+\eta
\norm[\L^2(0, t, \bfL^2(\Omega))]{\nabla u^i}^2
)\\&+(\Lambda_T + \sup_{A_i(T, C_{*})}
\vert \sigma^i \vert )
(t
+ C_{\eta}\norm[\L^2({{\mathcal{Q}}_t})]{u^i}^2
+\eta
\norm[\L^2(0, t, \bfL^2(\Omega))]{\nabla u^i}^2).
\end{aligned}
\end{equation}
Now, inequality \eqref{eq:GE17_3}
follows by choosing \(\eta > 0\)
small enough in \eqref{gr} and Gr\"onwall lemma. 
\end{proof}

\begin{corollary}
\label{GE:Cor1}
Let \(T>0\) be fixed. 
Assume that \(\emph{(A--1)--(A--3)}\) hold true. Assume that \(u_0\in\mathbf{L}^2(\Omega)\)
and \(u^i\geq 0\) for any \(i\in \{1, \ldots, n\}\).
Then, 
the problem
\((\mathcal{P}_{T})\)
admits exactly one solution.
Moreover, \(u(t)\geq 0\)
for any \(t\in [0, T]\).
\end {corollary}
\begin{proof}
As quoted in Theorem
\ref{Cor6}, the time existence
\(t_0\) is a function of \(\norm[\bfL^2(\Omega)]{u_0}\)
only. Due to Lemma
\ref{GE:Lemma3}, inequality 
\eqref{eq:GE17_3}, global  well-posedness follows.
\end{proof}

%%%%%%%%%%%%%%%%%%%%%%%%%%%%%%%%%%%%%%%%%%%%%%%%%

\section{Trace integrals inequalities.}
\label{Estimates_trace}

Our goal is to prove that Corollary 
\ref{GE:Cor1} holds true under a relaxed
assumption (A--2). This shall be done in Section
\ref{lasection} below. Since we argue 
by density, we first
have to determined the relevant estimates 
for the trace terms.
Our trace integral estimates 
(cf.~Lemma \ref{Lm7}) are 
consequences of a simple continuity 
lemma (see Lemma \ref{Lm1007} below).
Since in the sequel we loose an arbitrary 
small order 
of derivation by the use of the Aubin-Lions
lemma, we introduce a somewhat larger space than
\(\E_T\). 
For \(t\in(0, T)\) and \(\alpha \geq 0\), define
\begin{equation}
\label{spacal}
\E_t^{\alpha}
\eqldef \L^{\infty}(0, t, \mathbf{H}^{-\alpha}(\Omega))
\cap \L^2(0, t, \mathbf{H}^{1-\alpha}
(\Omega)) \quad\text{and}\quad
{\dot{\E}}_t\eqldef \bigcap_{0<\alpha \leq 1}\E_t^{\alpha}. 
\end{equation}
The space \({\dot{\E}}_t\) is endowed
with its natural Fr\'echet structure.
In particular, \(f_n\underset{n\rightarrow \infty}{\longrightarrow}  f\)
in \({\dot{\E}}_t\) iff 
\(f_n\underset{n\rightarrow \infty}{\longrightarrow} f\)
in all the \(\E_t^{\alpha}\), \(\alpha\in (0, 1)\).
By interpolation, for any
\(s\in[2, \infty)\) and 
\(r\in(0, 1)\), we have 
\begin{equation}
\label{spacil}
\L^{\infty}(0,T;\L^2(\Omega))\cap
\L^2(0,T;\H^{1}(\Omega))= \E_T^0
\inj 
{\dot{\E}}_T\inj \L^s(0,T;\L^2(\Omega))\cap
\L^2(0,T;\H^{r}(\Omega)).
\end{equation}
Until this end of the paper, we always 
abridge the notation \(\E_T^0\)
in \(\E_T\) (see \eqref{Norme}) 
and still denote by \(\mathcal{C}_{T, R}\) the closed ball
with radius \(R>0\) in \(E_T\).
\begin{lemma}
\label{Lm1007}
Let \(T>0\). We assume that
$ 1\leq p<2+\frac2d$ and $m\in[1,\infty[$.
Then, there exists \(\alpha = \alpha_{p, m}\in (0, 1)\) and \(C = C_{p, m}>0\) 
such that for any \((v, \bar{v})\in {\dot{\E}}_T\times 
\dot{\E}_T\), we have
\begin{subequations}
\begin{align}
\label{eq:Loc15}
&\norm[\L^p(\Gamma_T)]{v}\leq
C\norm[\E_T^{\alpha}]{v},\\&
\label{eq:Loc15_10}
\norm[\L^m(\Gamma_T)]{\mathcal{V}-\bar{\mathcal{V}}}
\leq
C\norm[\E_T^{\alpha}]{v-\bar{v}},
\end{align}
\end{subequations}
with \(\mathcal{V}(t)\eqldef\mathcal{B}(t,v(t))\) and 
\(\bar{\mathcal{V}}(t)\eqldef\mathcal{B}(t,\bar{v}(t))\) 
for almost every \(t\in[0,T]\). If moreover, we assume that
\(v\in\mathcal{C}_{T,R}\) (\(R>0\)) then
\begin{equation}
\label{eq:Loc15_11}
\norm[\L^m(\Gamma_T)]{\mathcal{V}}
\leq C_{R,T}.
\end{equation}
\end{lemma}
\begin{proof} 
According to H\"older's inequality, it's
enough to prove \eqref{eq:Loc15} for \(p\in]2, 2+\frac2d[\).
 We first assume that 
\(p\in ]2,\infty)\).
Let \(s\in [1,\infty)\),  \(\theta\in]0,1[\),
\(r\in ]0,1[\) and \(q\in ]0,1[\) such that
\([\L^s(0,T;\L^2(\Omega)),\L^2(0,T;\H^{r}
(\Omega))]_{\theta}=\L^{p}(0,T;\H^{q}(\Omega))\),
which is
\begin{subequations}
\label{eq:Loc16}
\begin{align}
\label{eq:Loc16_1}
&\frac{1}p
=\frac{1-\theta}s+\frac{\theta}2,\\&
\label{eq:Loc16_2}
q=\theta r,
\end{align}
\end{subequations}
where \([\cdot, \cdot ]_{\theta}\) 
denotes
the holomorphic interpolation fonctor 
(see \cite[p.~107]{BeLo76ISSV}). Note that the existence of such 
\(s, \theta, r, q\) is granted by
the condition \(p\in ]2,\infty)\). Now, 
for any \(v \in {\dot{\E}}_T\), the interpolation
and the Young inequalities give
%\begin{equation}
%\label{eq:Loc17}
%\begin{aligned}
\begin{align}
\norm[\L^p(0,T;\H^{q}(\Omega))]{v}
&\leq C\norm[\L^s(0,T;\L^2(\Omega))]{v}^{1-\theta}
\norm[\L^2(0,T;\H^{r}(\Omega))]{v}^{\theta}\nonumber\\&
\leq C\big(\norm[\L^{s}(0,T;\L^2(\Omega))]{v}+
\norm[\L^2(0,T;\H^{r}(\Omega))]{v}\big)
\nonumber\\&
\leq C\norm[\E_T^{\alpha}]{v}\label{eq:Loc17}
\end{align}
%\end{aligned}
%\end{equation}
for some \(\alpha = \alpha_{s, r}
\in (0,1)\)
(see \eqref{spacil}).
We now turn to determine the best
exponents \(p\) and \(q\)  in \eqref{eq:Loc17}. 
Notice that a limiting value for 
\(\theta\) in \eqref{eq:Loc16_1}
is \(\theta = \frac2p\). Hence
\(q= \frac2p\) is the corresponding
limiting value
for \(q\) in \eqref{eq:Loc16_2}.
Therefore, we can take
\(q=\frac{p}2 - \eta\), with 
\(\eta >0\) arbitrary small. 
Finally, we have
\begin{equation}
\label{blup}
v\in \L^p(0, T, 
\H^{2/p - \eta}(\Omega)).
\end{equation}
We now restrict to
\(1 \leq p < 2+\frac2d\) and  
write \(p^{-1} = \frac{d}{2(d+1)}+
\delta\) with \(\delta > 0\).
Choose \(\eta = \frac{(1+d)\delta}2\).
With these notations, we easily
compute
\begin{equation}
\label{blip}
\begin{aligned}
\dfrac{2}{p}-\eta -
d\Bigl(\dfrac{1}{2}-\dfrac{1}{p}\Bigr)
=
\dfrac{1}{p} + 
\eta.
\end{aligned}
\end{equation}
It follows from \eqref{blip}, Sobolev 
injections and trace lemmas that
\begin{equation*}
\L^p(0, T, 
\H^{2/p - \eta}(\Omega))
\inj 
\L^p(0, T, 
\W^{1/p + 
\eta,p}(\Omega))
\rightarrow 
\L^p(0, T, 
\L^{p}(\Bd)),
\end{equation*}
where the last arrow is also continuous.
Together with \eqref{blup} and \eqref{eq:Loc17}, it proves
\eqref{eq:Loc15}.

We now prove \eqref{eq:Loc15_10}. 
According to \eqref{eq:Loc7_2} and  \eqref{spacil}, we see that 
\begin{equation}
\label{eq:Loc15_13}
\norm[\L^m(0,T;\W^{1,\infty}(\Omega))]{\mathcal{V}-\bar{\mathcal{V}}}\leq
C_T\norm[\L^m(0,T;\L^2(\Omega))]{v-\bar{v}}\leq
C\norm[\E_T^{\alpha}]{v-\bar{v}}
\end{equation}
 for some \(\alpha = \alpha_{m}\).
 Now, inequality \eqref{eq:Loc15_10} follows from 
 \eqref{eq:Loc15_13} and a trace lemma.
 The proof of inequality \eqref{eq:Loc15_11}
is omited.
\end{proof}
In the sequel, we denote \(\mathcal{H}\eqldef 2+\frac2d\).
The following technical lemma will be used in the proof of the general existence theorem. 
It has essentially the same meaning as Lemma
\ref{Lm1007}, i.e.  boundary integrals of 
\( \vert u^i \vert^{p}\) can be bounded by
(functions of) \( \norm[\E_T]{u}\) 
for \(0\leq p <2+\frac2d\).
\begin{lemma}
\label{Lm7}
Assume that conditions
{\emph{(A--1)}} is satisfied. 
 Assume 
that  \(a\geq 0\), \(b\geq 0\)
\(1\leq a + b< \mathcal{H}\)
and \(\theta \geq 0\).
 Let 
 \(R>0\). There exists 
two constants \(C=C_{T, R, \Vert \mathcal{B} \Vert, a, b ,\theta}>0\) 
and \(\alpha =\alpha_{\theta, b}
\in (0, 1)\)
such that, for any \(u=(u^1,\ldots,u^n)\in\mathcal{C}_{T,R}\) and
\(\bar{u}=(\bar{u}^1,\ldots,\bar{u}^n)\in\mathcal{C}_{T,R}\),
we have
\begin{subequations}
\label{eq:Loc14}
\begin{align}
\label{eq:Loc0014}
&\int_{\Gamma_t}
(1+\abs{u^i}^{a})
\abs{u^j-\bar{u}^j}^b 
\abs{\mathcal{V}-
\bar{\mathcal{V}}}^{\theta}
\dd\mu\dd\tau
\leq
C\norm[\E_T^{\alpha}]{u-\bar{u}}^{b+\theta},
\\&
\label{eq:Loc014}
\int_{\Gamma_t}
(1+\abs{u^i}^{a})
\abs{u^j}^b 
\abs{\mathcal{V}}^{\theta}
\dd\mu\dd\tau
\leq
C,
\end{align}
\end{subequations}
for any \((i, j) \in \{1, \ldots, n\}^2\)
and \(t\in(0, T)\). 
As usual, we have written
\(\mathcal{V}(t)= \mathcal{B}(t, u(t))\) and
\(\bar{\mathcal{{V}}}(t)= \mathcal{B}(t, \bar{u}(t))\)
for a.e \(t\in]0,T]\). 
\end{lemma}
\begin{proof}
In this proof, we denote with a prime a conjugate exponent.
It is enough to prove inequality \eqref{eq:Loc14} for \(t=T\) and to estimate 
the integral 
\begin{equation*}
\mathcal{I}(a, b, \theta)
= \int_{\Gamma_T}
\abs{u^i}^{a}
\abs{u^j-\bar{u}^j}^b 
\abs{\mathcal{V}-
\bar{\mathcal{V}}}^{\theta}
\dd\mu\dd\tau.
\end{equation*}
We restrict to the case 
\(a > 0\) and \(b > 0\),
since the cases 
\(a = 0\) or \(b=0\) are easier.
Let \(\epsilon >0\)
such that 
\(1\leq (1+\epsilon)(a+b)< \mathcal{H}\).
The H\"older inequality leads to
\begin{equation*}
\mathcal{I}(a, b, \theta)
\leq\norm[\L^{(1+\epsilon)'\theta
}(\Gamma_T)]{\mathcal{V}
-\bar{\mathcal{V}}}^{\theta}
\mathcal{I}((1+\epsilon)a, (1+\epsilon)b, 0))^{\frac{1}{1+\epsilon}}.
\end{equation*}
Hence, by using \eqref{eq:Loc15_10}, we find
\begin{equation}
\label{zerof}
\mathcal{I}(a, b, \theta)
 \leq C 
 \norm[\E_T^{\alpha}]{u-\bar{u}}^{\theta} \mathcal{I}((1+\epsilon)a, 
 (1+\epsilon)b, 0))^{\frac{1}{1+\epsilon}}.
\end{equation}
Therefore, setting 
\(a_1\eqldef(1+\epsilon)a>0\) and
\(b_1 \eqldef (1+\epsilon)b> 0\), and recalling
that \(1\leq a_1+b_1 < \mathcal{H}\),
we just
have to prove 
\eqref{eq:Loc0014}
for 
$
\mathcal{I}(a_1, b_1, 0)$,
or simply 
$
\mathcal{I}(a, b, 0)$
which we now estimate.
Since \(a>0\), $b>0$ and
$1\leq a+b<\mathcal{H}$, we have $\frac{a\mathcal{H}}{\mathcal{H}-b}<\mathcal{H}$.
Let $\gamma\in \bigl]\max\bigl(1, 
\frac{a\mathcal{H}}{\mathcal{H}-b}\bigr), \mathcal{H}\bigr[$. For 
\(u\in \mathcal{C}_{T, R}\), inequality
\eqref{eq:Loc15} and \eqref{spacil} ensure that 
\(u\in\L^{\gamma}((0, T)\times \Bd)\)
with 
\begin{equation}
\label{oubli}
\norm[\L^{\gamma}(\Gamma_T)]{u^j}\leq C_{R, T}.
\end{equation}
Set \(q = \frac{\gamma}{a} >\frac{\mathcal{H}}{\mathcal{H}-b} > 1\). We have 
\begin{equation}
\label{zetrof}
q' < \Bigl(\frac{\mathcal{H}}{\mathcal{H}-b}\Bigr)' = \frac{\mathcal{H}}{b}.
\end{equation}
By H\"older inequality
\begin{equation}
\label{zettrof}
\mathcal{I}(a, b, 0)
 \leq 
 \norm[\L^{\gamma}(\Gamma_T)]{u^i}^a
  \norm[\L^{bq'}(\Gamma_T)]{u^j-\bar{u}^j}^b,
\end{equation}
with a slight abuse of notation in the case
\(0< bq'<1\). Appealing to \eqref{zetrof} and  \eqref{eq:Loc15} (and H\"older inequality
in the case \(0< bq'<1\)),
we find that
\begin{equation}
\label{zetttrof}
  \norm[\L^{bq'}(\Gamma_T)]{u^j-\bar{u}^j}^b
  \leq C
  \norm[\E_T^{\alpha}]{u-\bar{u}}^b.
\end{equation}
The estimate on 
\(\mathcal{I}(a, b, 0)\) follows from
\eqref{oubli},
\eqref{zettrof} and \eqref{zetttrof}.
The proof of \eqref{eq:Loc014} is similar.
\end{proof}
Our new assumptions are motivated by Lemma  \ref{Lm7}. 
Assumptions (A--1) and (A--3) are not modified,
while assumption (A--2) becomes
\begin{enumerate}[(\text{A}--4)]
\item The fluxes \(\sigma^i:  [0,\infty)\times\Bd\times \Er\times\Er
\rightarrow\Er\) are measurable, locally bounded functions. Moreover,
there exists \(\rho\in [0, 1+\frac2{d})\) such that
\begin{equation}
\label{eq:015}
\begin{aligned}
&\forall M>0,\quad\exists K_M>0: \\
&\forall (t,x)\in[0,M]\times\Bd,\quad
\forall (v,\psi)\in  \Er\times[-M,M], \quad\forall 
(\bar{v},\bar{\psi})\in \Er\times[-M,M]:\\&
\abs{\sigma^i(t,x,v,\psi)-\sigma^i(t,x,\bar{v},\bar{\psi})}\\&
\leq K_M
((1+\abs{v}^{\rho}+\abs{\bar{v}}^{\rho})
\abs{v-\bar{v}}+
(1+\abs{v}^{\rho+1}+
\abs{\bar{v}}^{\rho+1})
\abs{\psi-\bar{\psi}}).
\end{aligned}
\end{equation}
\end{enumerate}
As an immediate consequence of Lemma \ref{Lm7}, we have
\begin{corollary}
\label{fluxls}
Let \(T>0\) and \(R>0\). Assume that the conditions
{\emph{(A--1)}} and {\emph{(A--4)}} are satisfied. Then

\begin{enumerate}[(i)]

\item For any \(1\leq s < \frac{\mathcal{H}}{\rho+1}\), 
there exist \(\alpha =\alpha_{s(\rho+1)}
\in(0, 1)\) such that the application \(\mathcal{G}:
\mathcal{C}_{T, R}
\rightarrow \bL^s(\Gamma_T)\) with
\((\mathcal{G}(v))(t, x) \eqldef
\sigma\big(t, x, v(t, x), \mathcal{B}(t,v(t))(x)\big)\)
is well defined and Lipschitz continuous
for \(\mathcal{C}_{T, R}\)
endowed with the \(\E_T^{\alpha}\)
norm.

\item Assume that \(\rho\in[0, \frac{2}{d})\). For any $u\in \mathcal{C}_{T, R}$,
$\bar{u}\in \mathcal{C}_{T, R}$,
$\mathcal{V}(t) \eqldef \mathcal{B}(t, u(t))$ and 
$\bar{\mathcal{V}}(t) \eqldef \mathcal{B}(t,\bar{u}(t))$ for a.e $t\in(0, T)$, we have
\begin{equation*}
\begin{aligned}
&\int_{\Gamma_t}
\vert
\sigma(\tau, x, u(\tau, x), 
\mathcal{V}(\tau, x))
-
\sigma(\tau, x, \bar{u}(\tau, x), 
\bar{\mathcal{V}}(\tau, x))
\vert_1
\vert
u -
\bar{u}(\tau, x)
\vert_1
\dd\mu \dd\tau\\
&\leq C\norm[\L^2(0,t;\mathbf{L}^2(\Omega))]{u-\bar{u}}^{2}+\eta
\norm[\E_t]{u-\bar{u}}^{2}.
\end{aligned}
\end{equation*}

\end{enumerate}

\end{corollary}
 \begin{proof}
 Property {\emph{(i)}} is a direct
 consequence of the property
 (A--4) and Lemma \ref{Lm7}. Similarly, {\emph{(ii)}} follows from  the property
 (A--4), the inequality \eqref{eq:Loc0014}
 and, for \(\alpha \in (0, 1)\), the inequality
\begin{equation*}
\norm[\E_t^{\alpha}]{\cdot}^2\leq C_{\alpha, \eta}\norm[\L^2(0,t;\mathbf{L}^2(\Omega))]{\cdot}^{2}+\eta
\norm[\E_t]{\cdot}^{2}.
\end{equation*} 
  \end{proof}
  
%-------------------------------------------------------------------------------------------------------------------------------------  
  
\section{Global existence: the general case}
\label{lasection}

Assume that conditions (A--1), (A--3) and (A--4)
are satisfied for the fluxes 
\(\sigma^i\). We still denote by 
\(\Lambda_{T}\) and  
\(k_j\), \(j=1,\ldots, n\), the constants appearing in the condition
(A--3).
In order to apply Corollary~\ref{GE:Cor1}, we define a family 
of functions \(\sigma^i_p:  [0,\infty)\times\Bd\times \Er\times\Er
\rightarrow\Er\), \(p\in\mathbb{N}^*\), 
\(i\in\{1,\ldots, n\}\) endowed with conditions 
(A--2) and (A--3). Let \(h\in\mathscr{D}(\mathbb{R})\) with
\(h(x)=1\) for \(\vert x \vert\leq 1\),  
\(h(x) = 0\) for \(\vert x \vert\geq2\)
 and \(0 \leq h \leq 1\). For any  
\((t, x, v, \psi)\in [0,\infty)\times\Bd\times \Er\times\Er
\rightarrow\Er\) set 
 \begin{equation*}
 \sigma^i_p(t, x, v, \psi) 
 = \sigma^i(t, x, v, \psi)h\Bigl(\frac{v}{p}\Bigr). 
 \end{equation*}
As easily checked, the function \(\sigma^i_p\)
satisfies the two conditions (A--2) and (A--3),   
with constants \(\Lambda_{T}^p = \Lambda_{T}\) and  
\(k_j^p = k_j\)
\((j=1,\ldots, n)\) independent of  \(p\).  Moreover, for any 
\(p\in \{1, \ldots, n\}\), we have
\(\vert \sigma_p \vert
\leq \vert \sigma \vert\).

Let now \(T>0\), and let 
\(u_0\in \bfL^2(\Omega)\) with 
 \(u_0^i\geq 0\) for \(i\in\{1,\ldots,n\}\).
For any \(p\in\mathbb{N}^*\), Corollary \ref{GE:Cor1}
asserts the existence of a unique solution \(u_p\in \E_T\) to problem 
\((\mathcal{P}_T)\). 
Now, it follows from the Lemma \ref{GE:Lemma3}
that the sequence
\(\{\Vert u_p \Vert_{\E_T}\}
_{p\in \mathbb{N}^*}\) is bounded.
In the sequel, we denote by \(R\eqldef\sup_{p\in\mathbb{N}^*}
\Vert u_p \Vert_{\E_T}<\infty\).
Thus, for any \(p\in\mathbb{N}^*\), we have
\begin{equation}
\label{Interieur}
u_p\in\mathcal{C}_{T, R},
\end{equation}
and by Lemma \ref{Lm4} and a trace lemma, we get
\begin{equation}
\label{Trace}
\norm[\L^{\infty}(0, T,
\mathbf{W}^{1, \infty}(\Omega)\cap
\mathbf{W}^{2, 1}(\Omega))]{\mathcal{V}_p}
+\norm[\L^{\infty}(\Gamma_T)]{\mathcal{V}_p}
\leq C_{R,T}.
\end{equation}
This allows us to use all the previous results
of the paper. The rest of this section is devoted to the
proof of the convergence of 
\(\{u_p\}_{p\in\mathbf{N}^*}\) towards an exact solution.
\begin{lemma}
\label{AubinLions}
Under the assumptions
\emph{(A--1)}, \emph{(A--2)} and \emph{(A--4)}, and with the previous
notations, there exists \(u\in\mathcal{C}_{T, R}\)
such that, extracting if necessary a subsequence
\begin{subequations}
\begin{align}
\label{cuve}
&u_p\underset{p\rightarrow \infty}{\longrightarrow} u  
\text{ strongly in }
 {\dot{\E}}_T \text{ and }\C^0(0, T, 
\mathbf{H}^{-1/4}(\Omega)),\\
\label{cave}
&u_p\underset{p\rightarrow \infty}{\longrightharpoonup} u  
\text{ weakly in }
\L^2(0, T, \mathbf{H}^1(\Omega))\text{ and }
\text{ weakly--}\star \textrm{ in }
\L^{\infty}(0, T, \mathbf{L}^2(\Omega)).
\end{align}
\end{subequations}
  \end{lemma}
\begin{proof}
Since \(u_p\) satisfies equation 
\eqref{eq:011} in the sense of distributions, we deduce from 
\eqref{Interieur} and \eqref{Trace} that 
\begin{equation}
\label{timmes}
{\Bigl\{\frac{\dd u_p}{\dd t}\Bigr\}}_{p\in\mathbf{N}^*}
\text{ is  bounded in }
\L^2(0, T, \mathbf{H}^{-1}(\Omega)).
\end{equation} 
With \eqref{Interieur}, and using the Aubin-Lions lemma and a 
diagonal process, we extract 
from \(\{u_p\}_{p\in\mathbf{N}^*}\)
a 
converging  (and not relabeled) 
subsequence in 
\({\dot{\E}}_T\).
Still by the Aubin-Lions lemma, we can also
assume that  \(\{u_p\}_{p\in\mathbf{N}^*}\)
converges strongly in \(\C^0(0, T, 
\mathbf{H}^{-1/4}(\Omega))\). Properties
\eqref{cave} and 
\(u\in\mathcal{C}_{T, R}\) follow from \eqref{Interieur}.
\end{proof}
We now introduce
a space of test functions compatible with
the boundary conditions.
Until the end of the paper,
for the sake of clarity,
we sometimes go back to the notation 
\(v_{|_{\partial\Omega}}(t)\) for
the trace of \(v(t)\) 
on \(\partial\Omega\).
For \(T>0\) and \(0\leq \rho \leq 
\mathcal{H}-1\),
we set
\begin{equation}
\label{bto}
\begin{aligned}
\mathbf{b}(T, \rho) \eqldef &\Big{\{}v\in 
\L^2(0, T, \bH^1(\Omega)) \text{ 
such that 
there exists } r>\frac{\mathcal{H}}{\mathcal{H}-(\rho+1)}\\
&\text{ depending of }v, \text{ with } 
v_{|_{\partial\Omega}}
\in \L^r(\Gamma_T)\Big{\}}.
\end{aligned}
\end{equation}
Notice that for 
\(0\leq \rho<\mathcal{H}-2\), we have
\(\frac{\mathcal{H}}{\mathcal{H}-(\rho+1)} < \mathcal{H}\). Hence, by Lemma
\ref{Lm1007}, we have \(\E_T\subset \mathbf{b}(T, \rho)\)
for \(0\leq \rho < 2/d\). It follows that
\begin{equation}
\label{bta}
\H^1([0, T]\times\Omega)\subset \mathbf{b}(T, \rho)
\quad\text{for}\quad 0\leq \rho < 2/d.
\end{equation}
Notice also that, for 
\(0\leq \rho < \frac{\mathcal{H}-2}2\),
we have \(\frac{\mathcal{H}}{\mathcal{H}-(\rho+1)}< 2\).
Hence, \(\L^2(0, T, \bH^1(\Omega))\subset \mathbf{b}(T, \rho)\), 
and since the opposite inclusion is also
true, we have 
\begin{equation}
\label{bti}
\mathbf{b}(T, \rho)=
\L^2(0, T, \bH^1(\Omega)) \quad\text{for}\quad 0\leq \rho < 1/d.
\end{equation}
Last, 
the inclusion 
\begin{equation}
\label{btz}
\C^{1}(\mathcal{Q}_T)\subset \mathbf{b}(T, \rho)
\quad\text{for}\quad 0\leq \rho < 1+2/d
\end{equation}
holds true. We are ready to prove our existence
theorem.
\begin{theorem}
\label{Thetheorem}
Let \(T>0\) be fixed. 
Let \(0\leq \rho < 1+\frac2d\), and
assume that {\emph{(A--1)}}, 
{\emph{(A--3)}} and {\emph{(A--4)}} hold true.  
Let \(u_0\in\bL^2(\Omega)\) with 
 \(u_0^i\geq 0\), \(i\in\{1,\ldots,n\}\).\\ 

\begin{enumerate}[(i)]

\item The problem
\begin{equation*}
(\mathcal{R}_{T})\hspace{1em}
\begin{cases}
\text{Find }u\in \L^2(0,T;\bfH^1(\Omega))
\cap\L^{\infty}([0,T];\bL^2(\Omega)) 
\text{ such that for any}\\ 
w\in\C^1(\mathcal{Q}_T),
\text{ and a.e } t\in (0, T)\\ 
\displaystyle{-
\int_{\mathcal{Q}_t}
u\cdot\partial_t w
\dd x
\dd\tau
+\int_{\mathcal{Q}_t} (\nabla u+
u_{\alpha}\otimes\nabla\mathcal{V})
(\tau, x):\nabla w(\tau, x)\dd x\dd\tau}\\
\displaystyle{=
\int_{\Gamma_t}
\sigma(\tau, x, u(\tau, x),
\mathcal{V}(\tau, x))\cdot w(\tau, x)
\dd\mu \dd\tau+\int_{\Omega}u_0\cdot w(0)\dd x-\int_{\Omega}u(t)\cdot w(t)\dd x,}
\end{cases}
\end{equation*}
\big(with \(\mathcal{V}(\tau) = 
\mathcal{B}(\tau, u(\tau))\) for a.e
\(\tau\in(0, T)\)\big)
admits at least one solution.
In the case \(0\leq \rho< \frac2d\),
one can choose any \(w\in \H^1\big(
(0, T)\times \Omega\big)\) as a test function.

\item Assume that \(0\leq \rho< \frac1d\). 
Then, the problem \((\mathcal{P}_T)\)
admits exactly one solution.
\end{enumerate}
\end{theorem}
 
\begin{proof} 
We still denote by \(u\) and \(u_p\) the functions of
Lemma \ref{AubinLions}. 
As usual, we write 
\(\mathcal{V}(t) =
\mathcal{B}(t, u(t))\) and 
\(\mathcal{V}_p(t) =
\mathcal{B}(t, u_p(t))\)
for a.e \(t\in (0, T)\).
We must prove that \(u\) is a solution
of problem \(\mathcal{R}_T\) or
\(\mathcal{P}_T\).\\

\noindent{\emph{(i)}} \textbf{Existence in the case }\(\mathbf{0\leq \rho< 1+\frac2d.}\)
 
Let \(w\in\mathbf{b}(T, \rho)\).
Since the sequence
\(\{u_p\}_{p\in\mathbf{N}^*}\)
converges
in \({\dot{\E}}_T\) (see Lemma \ref{AubinLions}),
we deduce from
\eqref{spacil} its convergence
in \(\L^4(0, T ,\L^2(\Omega))\).
Therefore, by Lemma
\ref{Lm4}, we also have 
\(\nabla\mathcal{V}_p\underset{p\rightarrow \infty}{\longrightarrow}
\nabla\mathcal{V}\) in 
\(\L^4(0, T ,\L^{\infty}(\Omega))\).
Last, recall that \(\{u_p\}_{p\in\mathbf{N}^*}\)
converges weakly
in \(\L^2(0, T, \mathbf{H}^1(\Omega)))\). 
As a consequence of these convergences
\begin{equation}
\label{3}
\int_{\mathcal{Q}_T}
\big(\nabla u_p+(u_p)_{\alpha}\otimes
\nabla\mathcal{V}_p\big):\nabla w
\dd x \dd\tau
\underset{p\rightarrow \infty}{\longrightarrow}
\int_{\mathcal{Q}_T}
\big(\nabla u+u_{\alpha}\otimes
\nabla\mathcal{V}\big):\nabla w
\dd x\dd\tau.
\end{equation}
We now prove that 
\begin{equation}
\label{4}
D_p\eqldef\int_{\Gamma_T}
(\sigma_p(\tau, x, u_p(\tau, x)
, \mathcal{V}_p(\tau, x))-
\sigma(\tau, x, u(\tau, x)
, \mathcal{V}(\tau, x)))\cdot w(x)
\dd\mu \dd\tau\underset{p\rightarrow \infty}{\longrightarrow} 0.
\end{equation}
Since \(w\in\mathbf{b}(T, \rho)\),
we have  \(w_{|_{\partial\Omega}}\in
\L^r(\Gamma_T)\) for some 
\(r>\frac{\mathcal{H}}{\mathcal{H}-(\rho+1)}\). Hence, it is
enough to show
that 
\begin{equation}
\label{klou}
J_p\eqldef\int_{\Gamma_T}
\vert\sigma_p(\tau, x, u_p(\tau, x)
, \mathcal{V}_p(\tau, x))-
\sigma(\tau, x, u(\tau, x)
, \mathcal{V}(\tau, x))\vert_1^s
\dd\mu \dd\tau\underset{p\rightarrow \infty}{\longrightarrow} 0,
\end{equation}
for \(s=r' < \Big(\frac{\mathcal{H}}{\mathcal{H}-(\rho+1)}\Big)' = \frac{\mathcal{H}}{\rho+1}\).
Using \(\vert h \vert \leq 1\), we see that
\begin{equation}
\label{5'}
J_p \leq J_{p, 1}+J_{p, 2},
\end{equation}
where
\begin{subequations}
\begin{align}
&
J_{p, 1}\eqldef C_s 
\int_{\Gamma_T}
\vert\sigma(\tau, x, u_p(\tau, x)
, \mathcal{V}_p(\tau, x))-
\sigma(\tau, x, u(\tau, x)
, \mathcal{V}(\tau, x))\vert_1^s
\dd\mu \dd\tau,\\&
J_{p, 2}\eqldef
 C_s
\int_{\Gamma_T}
\vert 1-
h\big(u_p(\tau, x)/p\big)
\vert^s
\vert
\sigma(\tau, x, u(\tau, x)
, \mathcal{V}(\tau, x))\vert_1^s
\dd\mu \dd\tau.
\end{align}
\end{subequations}
Notice that 
\begin{align}
\label{6}
J_{p, 1}\underset{p\rightarrow \infty}{\longrightarrow}0,
\end{align}
due to \(1\leq s<\frac{\mathcal{H}}{\rho+1}\), 
Corollary~\ref{fluxls}{\emph{(i)}}, and 
the convergence of
\(\{u_p\}_{p\in\mathbf{N}^*}\)
in \({\dot{\E}}_T\). Next, 
since
\(1\leq s < \mathcal{H}\), invoking again
the convergence of
\(\{u_p\}_{p\in\mathbf{N}^*}\)
in \({\dot{\E}}_T\)
and \eqref{eq:Loc15},
we obtain the following convergence
\begin{equation}
\label{1'}
u_{p_{|_{\Bd}}}\underset{p\rightarrow \infty}{\longrightarrow}
u_{|_{\Bd}}
\quad\text{in}\quad\L^s(\Gamma_T).
\end{equation}
Extracting if necessary 
a subsequence, we get
\(u_{p{|_{\Bd}}} \underset{p\rightarrow \infty}{\longrightarrow} u_{|_{\Bd}}
\) a.e. Since \(h(0) = 0\), we finally obtain
\begin{equation}
\label{7}
\vert 1-
h\big(u_p(\tau, x)/p\big)
\vert^s
\vert
\sigma(\tau, x, u(\tau, x)
, \mathcal{V}(\tau, x))\vert_1^s
\underset{p\rightarrow \infty}{\longrightarrow} 0,
\end{equation}
for a.e \((\tau, x)\in\Gamma_T\). Moreover, 
\begin{equation}
\label{8}
\vert 1-
h\big(u_p(\tau, x)/p\big)
\vert^s
\vert
\sigma(\tau, x, u(\tau, x)
, \mathcal{V}(\tau, x))\vert_1^s
\leq
\vert
\sigma(\tau, x, u(\tau, x)
, \mathcal{V}(\tau, x))\vert_1^s,
\end{equation}
with \(\int_{\Gamma_T}
\vert\sigma(\tau, x, u(\tau, x)
, \mathcal{V}(\tau, x))\vert_1^s
\dd\mu \dd\tau
< \infty\) (see Corollary \ref{fluxls}).
From \eqref{7}, \eqref{8} and Lebesgue theorem,
we derive that \(J_{p, 2}\) tends to \(0\).
With \eqref{6}, \eqref{5'} and \eqref{klou}, this proves \eqref{4}.\\
Now, since the function
\(u_p\) is a solution of the problem 
(\(\mathcal{P}_{T}\)) with \(\sigma_p\)
in place of \(\sigma\), using 
\eqref{3} and \eqref{4} we see that
\begin{equation}
\label{9}
\begin{aligned}
\int_0^T
\Bigl{\langle}
\dfrac{\dd u_p(\tau)}{\dd t}, w(\tau)
\Bigr{\rangle} \dd\tau
\underset{p\rightarrow \infty}{\longrightarrow}&
\int_{\mathcal{Q}_T}
\big(\nabla u+u_{\alpha}\otimes
\nabla\mathcal{V}\big):\nabla w
\dd x\dd\tau\\&
+
\int_{\Gamma_T}
\sigma(\tau, x, u(\tau, x)
, \mathcal{V}(\tau, x))\cdot w(x)
\dd\mu \dd\tau.
\end{aligned}
\end{equation}
Starting with \(t\in [0, T]\) in place of \(T\),
we conclude that \eqref{9} holds true for any \(t\in [0, T]\).
Restricting to
\(w\in \C^1([0, T]\times\Omega)\) (see 
\eqref{btz}), we have, for any \(t\in (0, T)\) 
\begin{equation}
\label{mou}
\begin{aligned}
&\int_0^t
\Bigl{\langle}
\frac{\dd u_p}{\dd t}(\tau), w(\tau)
\Bigr{\rangle} \dd\tau
=
-\int_{\mathcal{Q}_t}
u_p\cdot\partial_tw
 \dd x\dd\tau
  -\int_{\Omega}u_0\cdot w(0)\dd x+\int_{\Omega}u_p(t)\cdot w(t)\dd x.
 \end{aligned}
 \end{equation}
In order to pass to the limit in \eqref{mou}
notice the convergence \(\int_{\Omega}u_p\cdot w\dd x
\underset{p\rightarrow \infty}{\longrightarrow}\int_{\Omega}u\cdot w\dd x\)
in \(\L^1(0, T)\), consequence of 
\(u_p\underset{p\rightarrow \infty}{\longrightarrow} u\) in \(\L^2(\mathcal{Q}_T)\) 
(see Lemma \ref{AubinLions}).
Hence, extracting if necessary a subsequence, we get
\begin{equation}
\label{miel}
\int_{\Omega}u_p(t)\cdot w(t)\dd x
\underset{p\rightarrow \infty}{\longrightarrow} \int_{\Omega}u(t)\cdot w(t)\dd x
\textrm{ for a.e }t \in (0, T).
\end{equation}
Appealing to \eqref{mou}, \eqref{miel} and \eqref{cave}
with \(t\) in place of \(T\),
we obtain
\begin{equation}
\label{milou}
\int_0^t
\Bigl{\langle}
\frac{\dd u_p}{\dd t}(\tau), w(\tau)
\Bigr{\rangle} \dd\tau
\underset{p\rightarrow \infty}{\longrightarrow}
-\int_{\mathcal{Q}_t}
u\cdot\partial_tw
 \dd x\dd\tau
  -\int_{\Omega}u_0\cdot w(0)\dd x+\int_{\Omega}u(t)\cdot w(t)\dd x
\end{equation}
for a.e \(t \in (0, T)\).
Appealing to \eqref{9}, it implies that
\((\mathcal{R}_T)\) holds true. The proof 
in the case \(0\leq \rho <\frac2d\) is 
similar (see \eqref{bta}). \\

\noindent {\emph{(ii)}} \textbf{Existence and uniqueness in the case} \(\bf{0\leq \rho
<\frac1d}\). We begin with the existence part. 
From one hand, since \(u_p\underset{p\rightarrow \infty}{\longrightarrow} u\) in
\(\L^2(\mathcal{Q}_T))\), we obtain the
convergence in \(\L^2(0, T,
\mathbf{\H}^{1}(\Omega)')\).
Hence we find \(\frac{\dd u_p}{\dd t}\underset{p\rightarrow \infty}{\longrightarrow}
\frac{\dd u}{\dd t}\) in 
\(\H^{-1}(0, T,
\mathbf{\H}^{1}(\Omega)')\). On the other hand, since 
\(0\leq \rho < \frac1d\) we have
\(\mathbf{b}(T, \rho) = 
\L^2(0, T,
\mathbf{\H}^{1}(\Omega))\), and by 
\eqref{9}, \(\bigl{\{}\frac{\dd u_p}{\dd t}\bigr{\}}_{p\in \mathbb{N}^*}\) 
converge towards some \(f\) in 
\(\L^2(0, T,
\mathbf{\H}^{1}(\Omega)')\) weak--\(\star\).
By identification, \(f = \frac{\dd u}{\dd t}\), and
the variational existence part follows from \eqref{9}. Next, \(u\in \C^0(0, T,
\mathbf{L}^2(\Omega))\) follows
classically from \(u\in \L^2(0, T,
\mathbf{H}^1(\Omega))\) and 
\(\frac{\dd u}{\dd t}\in \L^2(0, T,
\mathbf{H}^1(\Omega)')\). Last, writing 
\(u_p(0) = u_0\) and using \eqref{cuve}, we
get \(u(0) = u_0\).

For the uniqueness part, let \(u\in\E_T\),
 \(\bar{u}\in \E_T\) be two solutions
 of (\(\mathcal{P}_{T}\)) associated 
 with the same initial data. Denote by
 \(\mathcal{V}\) and 
 \(\bar{\mathcal{V}}\) the associated potentials.
 Notice that
 \(u-\bar{u}\in 
 \E_T\subset\mathbf{b}(T, \rho)\)
 by \eqref{bta}.
 Using \((\mathcal{P}_T)\), we derive the
 following energy estimate
 \begin{equation*}
 \begin{aligned}
&\norm[\E_t]{u-\bar{u}}^2
\leq C_{R, T}
\norm[\L^2(\mathcal{Q}_t)]{u-\bar{u}}^2\\
&+ C_{R, T}
\int_{\Gamma_t}
\vert
\sigma(\tau, x, u(\tau, x),
\mathcal{V}(\tau, x))
-
\sigma(\tau, x, \bar{u}(\tau, x),
\mathcal{V}(\tau, x))
\vert_1
\vert
u-\bar{u}
\vert_1
\dd\mu \dd\tau.
 \end{aligned}
 \end{equation*}
 Appealing to Corollary \ref{fluxls} \emph{(ii)}, we get, for
 \(\eta>0\) small enough, we obtain
 \begin{equation*}
 \norm[\E_t]{u-\bar{u}}^2
\leq C_{R, T}
\norm[\L^2(\mathcal{Q}_t)]{u-\bar{u}}^2,
\end{equation*}
and uniqueness result follows.
 \end{proof}
   \begin{remark}   
   Theorem \ref{Thetheorem}
   holds true with a general diffusion term 
   \((\eta_1 \Delta u_1, \cdots, \eta_n \Delta u_n)^{\tra}\)in place 
   of \(\Delta u\) (\(\eta_i>0\)). This will  implicitely
   be used in Section 
   \ref{zamples}
    \end{remark}
%%%%%%%%%%%%%%%%%%%%%%%%%%%%%%%%%%%%%%%%%%%%%%%%%

\section{Examples}
\label{zamples}

In this concluding section, we illustrate our setting by some realistic equations.
More precisely, we focus below on two examples coming from 
drift-diffusion and self-gravitation models. The first example deals with a 
drift-diffusion system with Robin boundary condition with application to a corrosion
model and the well-posedness follows from direct application of Theorem \ref{Thetheorem}.
In the second example,  the Theorem 
\ref{Thetheorem} is used in the proof as a mollifying frame. This
method could be used for more complicated systems.

\subsection{The drift-diffusion system coming from a corrosion model}

We consider first a drift-diffusion system subjected to Robin boundary conditions.  
Under appropriate assumptions on the data, existence of solutions comes
from Theorem~\ref{Thetheorem}. Then we illustrate this general setting with more
specific model, namely the corrosion modeling in a nuclear waste
repository (cf.~\cite{BaBoC*12NMSC,CHAVI}).

\subsubsection{The drift-diffusion system}
Assume that for any \(i\in\{1,\ldots, n\}\), \(\alpha_i\in\Er\), \(\beta_i\in\Er\), \(\theta_i\) 
a Borel measure on \([a,b]\) (\((a,b)\in\Er^2\)) and \(u_0^i\in\L^2(0,1)\) with \(u_0^i\geq 0\), 
\((\mathcal{A}_0,\mathcal{A}_1,\mathcal{V}_0,\mathcal{V}_1)\in\Er^4\)
and \(\zeta>0\).
Let us define \(f^i\in\C^1(\partial\Omega\times\Er\times [a,b])\) and 
\(g^i\in\C^1(\Er\times [a,b])\) satisfying the following assumptions:
\begin{enumerate}[(D--1)]
\itemsep0.2em

\item \(\forall (x,\phi,s)\in [0,1]\times\Er\times [a,b]:\; f^i(x,\phi,s)\leq 0\),

\item \(\exists R\in [0,\infty), \forall (v,s)\in[R,\infty)\times[a,b]:\; g^i(v,s)\geq 0\),

\item \(\forall (v,s)\in (-\infty,0]\times [a,b]:\; g^i(v,s)\leq 0\),

\item \(\forall (v,\bar{v}, s)\in \Er\times\Er\times [a,b]:\; 
\abs{g^i(v,s)-g^i(\bar{v},s)}\leq K(1+\abs{v}^{\rho}+\abs{\bar{v}}^{\rho})\abs{v-\bar{v}}\) for some
\(\rho\in[0,3)\).

\end{enumerate}
Set
\begin{equation}
\forall i\in\{1,\ldots, n\}:\; 
\sigma^i(t,x,v,\phi)\eqldef \int_{[a,b]} f^i(x,\phi,s)g^i(v,s)\dd\theta_{i}(s). 
\end{equation}
We consider the following drift-diffusion system:
\begin{subequations}
\label{drift_1}
\begin{align}
\forall i\in\{1,\ldots,n\}:\;
\partial_t u^i=\partial_x(\partial_xu^i+\alpha_i(u^i\partial_x \mathcal{V})),\quad(t,x)\in
(0,T)\times(0,1),\\
\partial_{xx}\mathcal{V}=\sum_{i=1}^n\beta^i u^i+\zeta,\quad(t,x)\in(0,T)\times(0,1),
\end{align}
\end{subequations}
together with boundary conditions
\begin{subequations}
\label{drift_2}
\begin{align}
\forall i\in\{1,\ldots,n\}:\;
-\big(\partial_x u^i+\alpha_i u^i\partial_x\mathcal{V}\big)(t,0)=\sigma^i(t,0,u^i(t,0),\mathcal{V}(t,0)),
\quad t\in[0,T],\\
\forall i\in\{1,\ldots,n\}:\;
\big(\partial_x u^i+\alpha_i u^i\partial_x\mathcal{V}\big)(t,1)=\sigma^i(t,1,u^i(t,1),\mathcal{V}(t,1)),
\quad t\in[0,T],\\
\mathcal{V}(t,0)+\mathcal{A}_0\partial_x\mathcal{V}(t,0)=\mathcal{V}_0,\quad t\in[0,T],\\
\mathcal{V}(t,1)+\mathcal{A}_1\partial_x\mathcal{V}(t,1)=\mathcal{V}_1,\quad t\in[0,T],
\end{align}
\end{subequations}
and initial conditions
\begin{equation}
\label{drift_3}
\forall i\in\{1,\ldots,n\}:\; u^i(0,x)=u_0^i,\quad x\in(0,1).
\end{equation}
We suppose now that \(1+\mathcal{A}_1-\mathcal{A}_0\neq 0\) and \(\varphi\in\L^1(0,1)\). 
Then the following problem:
\begin{equation}
\label{eq:16}
\begin{cases}
\text{Find }\mathcal{V}\in\W^{2,1}(0,1)\text{ such that }\\
\partial_{xx}\mathcal{V}=\varphi,\\
(\mathcal{V}+\mathcal{A}_0 \partial_x\mathcal{V})(\cdot,0)=\mathcal{V}_0
\quad\text{and}\quad(\mathcal{V}+\mathcal{A}_1 \partial_x\mathcal{V})(\cdot,1)=\mathcal{V}_1,
\end{cases}
\end{equation}
admits exactly one solution given by
\begin{equation}
\label{eq:17}
\mathcal{V}(x)=\int_0^1 G(x,y)\varphi(y)\dd y+\frac{x-\mathcal{A}_0}{1+\mathcal{A}_1-
\mathcal{A}_0}(\mathcal{V}_1-\mathcal{V}_0)+\mathcal{V}_0,
\end{equation}
where \(G\in\L^{\infty}(0,1;\W^{1,\infty}(0,1))\cap \C^0([0,1]\times[0,1])\)
is the Green kernel associated with problem \eqref{eq:16}.
We may observe that the function \(G\) is defined as follows:
\begin{equation*}
G(x,y)\eqldef\frac{(1+\mathcal{A}_1-x)(\mathcal{A}_0-y)}{1+\mathcal{A}_1-
\mathcal{A}_0}
\end{equation*}
for \(0\leq y\leq x\leq 1\) and \(G(x,y)=G(y,x)\)
for \(0\leq x\leq y\leq 1\). Notice that (A--2) and (A--4) follow from \eqref{eq:17} while
(A--1) comes from the assumptions (D--1)--(D--3).
Since it is quite a routine to verify that (A--1), (A--2) and (A--4) holds
true, the verification is let to the reader. Consequently, we may deduce from  
Theorem~\ref{Thetheorem} that \eqref{drift_1}--\eqref{drift_3} admits at least one solution 
\(u\eqldef (u_1,\ldots, u_n)^{\tra}\)
(in the sense of \((\mathcal{R}_T)\)) belonging to 
\(\L^2(0,T;\bH^1(0,1))\cap\L^{\infty}([0,T];\bL^2(0,1))\).

\subsubsection{A corrosion model}
We illustrate our setting by an example coming from the description of the
corrosion in a nuclear waste repository. 
Let \(u^1\), \(u^2\), \(u^3\) and \(\mathcal{V}\) be the electrons and cations densities,
oxygen vacancies and electrical potential,
respectively.  Following \cite{CHAVI}, we assume that the boundary conditions on 
\(u^1\), \(u^2\) and \(u^3\) have 
exactly the same form.
Let \(\zeta\) be the density charge in the host lattice and \(\lambda\) and \(\varepsilon\)
be two nonnegative constants such that \(\varepsilon\ll 1\). Set
\begin{equation}
\forall i=1,2,3:\;
\sigma^i(t,x,v,\phi)\eqldef
-(m_x^i\mathrm{e}^{-\gamma^i b^i_x\phi}+k_x^i\mathrm{e}^{\gamma^i a_x^i\phi})u^i+
m_x^i u^i_{\max}\mathrm{e}^{-\gamma^i b_x^i\phi},
\end{equation}
where \(m_x^i>0\), \(k_x^i>0\), \(a_x^i\in[0,1]\), \(b_x^i\in[0,1]\) and \(u_{\max}^i>0\) with
\(i=1,\ldots, 3\) 
%and \(u_{\max}^3=0\) 
and \(\gamma^1=-1\), \(\gamma^2=3\) and \(\gamma^3=1\).
The mathematical problem is formulated as follows:
\begin{subequations}
\label{Cor_1}
\begin{align}
\label{Cor_1_1}
\forall i=1,2,3:\;
\varepsilon^{2-i} \partial_t u^i=\partial_x(\partial_x u^i+\gamma^i
u^i\partial_x\mathcal{V}),\quad (t,x)\in\Er_*^+\times(0,1),\\
 \label{Cor_1_3}
-\lambda\partial_{xx}\mathcal{V}=\gamma^1 u^1+\gamma^2u^2
+\gamma^3u^3
+\zeta,
\quad (t,x)\in\Er_*^+\times(0,1).
\end{align}
\end{subequations}
The system \eqref{Cor_1} is endowed with the following boundary conditions
\begin{subequations}
\label{Cor_2}
\begin{align}
\label{Cor_2_1}
\forall i=1,2,3:\;
-(\partial_xu^i+\gamma^iu^i\partial_x\mathcal{V})(t,0)=\sigma^i(t,0,u^i(t,0),\mathcal{V}(t,0)),
\quad t\in\Er_*^+,\\
\label{Cor_2_4}
\forall i=1,2,3:\;
(\partial_x u^i+\gamma^iu^i\partial_x\mathcal{V})(t,1)=\sigma^i(t,1,u^i(t,1),\Psi-\mathcal{V}(t,1)),
\quad t\in\Er_*^+,\\
\label{Cor_2_3}
(\mathcal{V}-\mathcal{A}_0\partial_x\mathcal{V})(t,0)=\Delta\mathcal{V}_0,
\quad t\in\Er_*^+,\\
\label{Cor_2_6}
(\mathcal{V}-\mathcal{A}_1\partial_x\mathcal{V})(t,1)=\Psi-
\Delta\mathcal{V}_1,\quad t\in\Er_*^+,
\end{align}
\end{subequations}
and the following initial conditions
\begin{equation}
\label{Cor_3}
\forall i=1,2,3 :\; u^i(0,x)=u_0^i,\quad x\in(0,1).
\end{equation}
Here \(\Psi\) denotes a given applied potential, \(\Delta\mathcal{V}_i\) 
are the voltage drop parameters and \((\mathcal{A}_0, \mathcal{A}_1)\in \Er^2\). 
Furthermore, for any \(i\in\{1,2,3\}\), \(u_0^i\geq 0\) belongs to \(\L^2(0,1)\). 
For further explanations on this model, the reader is referred
to~\cite{BaBoC*12NMSC} as well as to the references therein.
Appealing to Theorem~\ref{Thetheorem} (or see Subsection \ref{lasection}),
we infer that \eqref{Cor_1}--\eqref{Cor_3} 
possesses exactly one solution \(u=(u^1,u^2,u^3)^{\tra}\) 
belonging to \(\L^2(0,T;\bH^1(0,1))\cap\L^{\infty}(0,T;\bL^2(0,1))\) for any \(T>0\).
In contrast with the former existence result present in \cite{CHAVI} \((n=2)\), our 
result holds true for \(n=3\) and even for an arbitrary number of species (\(n\in\En^*\)). 
Furthermore, the conditions on the {\emph{voltage drops}} and {\emph{other structural 
coefficients}} in Theorem 1.1 in \cite{CHAVI} have been removed.

Finally, as quoted above, we have assumed that the boundary conditions 
on \(u^1\), \(u^2\) and \(u^3\) have 
the same form.
Nevertheless, in the original
corrosion system depicted in
\cite{BaBoC*12NMSC}, this is not the case.
As easily verified, the second boundary
condition given therein
does not meet our assumption
(A--3). In consequence, it is 
unclear to us whether this
boundary condition is mathematically
sound.

\subsection{The self-gravitational system}
We consider the self-gravitational system described in \cite{BilNad02GEGS}.
Let \(u(t, x)\) be the evolution density  of identical attracting particles and 
\(\mathcal{V}(t, x)\) be the gravitational potential. The mathematical problem can be written
as follows:
\begin{subequations}
\begin{align}
\partial_t u=\nabla \cdot(\nabla u
+u \nabla\mathcal{V}),&\quad (t,x) \in(0,T)\times\Omega,\label{zample1}
\\
\Delta \mathcal{V} = u,&\quad (t,x) \in(0,T)\times\Omega,\label{zample2}
\end{align}
\end{subequations}
together with the boundary conditions on \((0, T)\times\Bd\)
\begin{subequations}
\begin{align}
 \label{zample3}
\frac{\partial u}{\partial \nu}
 +u
 \frac{\partial\mathcal{V}}{\partial \nu}
 = 0, &\quad(t,x)\in(0, T)\times\Bd,\\
\label{zample4}
\mathcal{V}= 0, &\quad(t,x)\in(0, T)\times\Bd, 
\end{align}
\end{subequations}
and with initial data
\begin{equation}
\label{zample1}
u(0) =u_0.
\end{equation}
Observe that the above system corresponds to \(n=1\), \(\sigma = 0\), \(k^i = 0\) and 
\(\Lambda_T = 0\).
In the sequel, we restrict to the case \(d=2\), and derive as in \cite{BilNad02GEGS} 
an existence result for a small \(\L^2\) initial data. The proof relies on 
a \(\L^2\) estimate on the function \(u\). Since we can use the Theorem \ref{Thetheorem} 
in order to get a global existence result for a mollified
system, it is enough to prove that the crucial \(\L^2\) estimate
holds uniformly true for the family of approximate solutions.

Remark that the resolvent \(\mathcal{B} = \Delta_D^{-1}\) of the Poisson-Dirichlet
problem on \((0, T)\times \Omega\), namely
\(\Delta(\mathcal{B}(f))= f\) and \(\mathcal{B}(f)_{\vert_{\Bd}} = 0\), 
do not fulfill the 
\(\L^1\) or the \(\W^{1, \infty}\) condition in (A--1).
Nevertheless, \(\mathcal{B}\) defines a continuous operator in
\(\L^p(\Omega)\rightarrow \W^{2, p}(\Omega)\) for any
\(1<p< \infty\). We regularize the operator \(\mathcal{B}\) in the
following way. Let \(\phi\in\mathscr{D}(\mathbb{R})\) be a density 
probability. For any \(p\in\mathbb{N}^*\), set \(\phi_p(x)=p\phi(px)\), \(x\in\Er\). 
Let also \(E\) be the extension by zero operator. 
Notice that 
\(E: \L^1(\Omega)\rightarrow\L^1(\mathbb{R}^2)\) is continuous. 
Moreover,
\(Ev\geq 0\) for any positive \(v\in \L^1(\Omega)\).
For  \(p\in\mathbb{N}^*\), define \(\mathcal{B}_p\) on \(v\in\L^1(\Omega)\) 
by \(\mathcal{B}_p(v) = \mathcal{B}((Ev\star\phi_p)_{\vert_{\Omega}})\).
As easily checked, for any \(1<q<\infty\),  \(\mathcal{B}_p: \L^1(\Omega)
\rightarrow \W^{2, q}(\Omega)\) continuously. In fact, for any \(v\in
\L^1(\Omega)\), we have
\begin{equation}
\label{bat}
\begin{aligned}
\norm[\W^{2, q}(\Omega)]{\mathcal{B}_{p}(v)}
&\leq C_{q}
\norm[\L^{q}(\Omega)]{(Ev\star\phi_{p})_{\vert_{\Omega}}}
\leq C_{q}
\norm[\L^{q}(\mathbb{R}^2)]{Ev\star\phi_{p}}\nonumber\\
&\leq C_{q}
\norm[\L^{1}(\mathbb{R}^2)]{Ev}
\norm[\L^{q}(\mathbb{R}^2)]{\phi_p}
\leq C_{p, q}
\norm[\L^{1}(\Omega)]{v}.
\end{aligned}
\end{equation}
From H\"older inequality and Sobolev embeddings, it follows
from this \(\L^1-\W^{2, q}\) continuity that the operator
\(\mathcal{B}_p\) satisfies the condition (A--1).
Therefore (see Theorem \ref{Thetheorem}),
for any \(p\in \mathbb{N}^*\), we can define \(u_p\in\E_T\)
as the solution of \((\mathcal{P}_T)\). It means that, 
\(\frac{\mathrm{d}u_p}{\mathrm{dt}}\in\L^2(0,T;(\bH^1(\Omega))') \), 
 \(u_p(0) = u_0\) and,  
for any  
\( w\in\L^{2}(0,T;\bH^1(\Omega))\)
\begin{equation}
\label{PTQ}
\begin{aligned}
&\int_0^T\Bigl{\langle}
\frac{{\dd}u_p}{{\dd}\tau}(\tau),w(\tau)\Bigr{\rangle}\dd\tau
+\int_{\mathcal{Q}_T}(\nabla u_p+
u_{p}\nabla\mathcal{V}_p)
(\tau, x)\cdot\nabla w(\tau, x)\dd xd\tau
\end{aligned}
\end{equation}
with \(\mathcal{V}_p(t) = 
\mathcal{B}_p(t, u_p(t))\) for a.e 
\(t\in(0, T)\).
\begin{theorem}
\label{example2}
Let \(\Omega\) be a smooth, bounded domain of \(\Er^2\). 
There exists \(\eta>0\)
such that, for any \(T>0\), any \(u_0\in\L^2(\Omega;\Er)\) with \(u_0\geq 0\)
and \(\norm[\L^1(\Omega)]{u_0}\leq\eta\), the problem
\begin{equation*}
\begin{cases}
\text{Find }u\in\C^0([0, T], \bL^2(\Omega))
\cap \L^2(0,T;\bH^1(\Omega))
\text{ with }
\frac{\mathrm{d}u}{\mathrm{dt}}\in\L^2(0,T;(\bH^1(\Omega))')\\ 
\text{ such that, for any }
w\in\L^{2}(0,T;\bH^1(\Omega)):\\
{\displaystyle{\int_0^T\Bigl{\langle}
\frac{{\dd}u}{{\dd}\tau}(\tau),w(\tau)\Bigr{\rangle}\dd\tau
+\int_{\mathcal{Q}_T}
(\nabla u+u\nabla\mathcal V)\cdot\nabla w\dd x\dd\tau}=0}\\
\text{with }\mathcal{V}\eqldef \Delta_D^{-1} u\text{ and }u(0)=u_0,
\end{cases}
\end{equation*}
admits at least one solution.
\end{theorem}
\begin{proof}
We mainly have to prove that the above sequence \(\{u_p\}_{p\in\En^*}\) is bounded in
\(E_T\). Since \(u_p\) satisfies the formulation \((\mathcal{P}_T)\), taking
\(w=u_p\) as a test function (\(p\in\En^*\)), we get 
\begin{equation}
\label{Ex2_1}
\frac12\norm[\L^2(\Omega)]{u_p(t)}^2+\norm[\L^2(0,T;\L^2(\Omega))]{\nabla u_p}^2
\leq\frac12\norm[\L^2(\Omega)]{u_0}^2-\int_{\mathcal{Q}_t}u_p\nabla\mathcal{V}_p
\cdot\nabla u_p\dd x\dd \tau.
\end{equation}
Notice that
\begin{equation}
\label{Ex2_2}
\int_{\mathcal{Q}_t}u_p\nabla\mathcal{V}_p\cdot\nabla u_p\dd x\dd\tau=
-\frac12\int_{\mathcal{Q}_t}u_p^2\Delta\mathcal{V}_p\dd x\dd\tau
+\frac12\int_{\mathcal{Q}_t} u_p^2\frac{\partial\mathcal{V}_p}{\partial \nu}\dd\mu\dd \tau.
\end{equation}
In order to remove the second term
in the right hand side of equality
\eqref{Ex2_2}, remark that
\(u_p\geq 0\). Hence, \(E u_p\geq 0\) so that \(\Delta\mathcal{V}_p=
(Eu_p\star \phi_p)_{|_{\Omega}}\geq 0\). Recalling the equality 
\({\mathcal{V}_p}_{|_{\partial\Omega}}=0\), we therefore obtain 
\(\frac{\partial\mathcal{V}_p}{\partial \nu}\geq 0\).
Finally, according to \eqref{Ex2_1} and \eqref{Ex2_2} and the definition 
of \(\mathcal{V}_p\), this leads to
\begin{equation}
\label{Ex2_3}
\frac12\norm[\L^2(\Omega)]{u_p(t)}^2+\norm[\L^2(0,t;\bfL^2(\Omega))]{\nabla u_p}^2
\leq
\frac12\norm[\L^2(\Omega)]{u_0}^2+\int_{\mathcal{Q}_t}u_p^2(Eu_p\star \phi_p)\dd x\dd\tau.
\end{equation}
It remains to estimate the last term on the right hand side of \eqref{Ex2_3}.  To this aim, 
we use the H\"older and the convolution inequalities to get
\begin{equation}
\label{Ex2_4}
\int_{\mathcal{Q}_t}u_p^2(Eu_p\star \phi_p)\dd x\dd\tau
\leq \norm[\L^3({\mathcal{Q}_t})]{u_p}^2\norm[\L^3({\mathcal{Q}_t})]{Eu_p\star\phi_p}
\leq\norm[\L^3({\mathcal{Q}_t})]{u_p}^3.
\end{equation}
By using \eqref{Ex2_4} and Gagliardo-Nirenberg inequality, it follows that there exists
a constant \(C_{\text{GN}}>0\) independent of \(p\in\En\) such that 
\begin{equation}
\label{Ex2_5}
\int_{\mathcal{Q}_t}u_p^2(Eu_p\star \phi_p)\dd x\dd\tau
\leq C_{\text{GN}}\int_0^t\norm[\L^1(\Omega)]{u_p(\tau)}\norm[\H^1(\Omega)]{u_p(\tau)}^2\dd\tau.
\end{equation}
Appealing to \eqref{eq:GE4_1} with \(k^i=0\), \(\sigma=0\) and \(\Lambda_T=0\), we see that
\begin{equation}
\label{Ex2_6}
\forall \tau\in[0,t]:\;\norm[\L^1(\Omega)]{u_p(\tau)}\leq \norm[\L^1(\Omega)]{u_0}.
\end{equation}
Hence \eqref{Ex2_3}--\eqref{Ex2_6} leads to
\begin{equation}
\label{Ex2_7}
\frac12\norm[\L^2(\Omega)]{u_p(t)}^2+\norm[\L^2(0,t;\bfL^2(\Omega))]{\nabla u_p}^2
\leq
\frac12\norm[\L^2(\Omega)]{u_0}^2+ C_{\text{GN}}\norm[\L^1(\Omega)]{u_0}
\int_0^t\norm[\H^1(\Omega)]{u_p(\tau)}^2\dd\tau.
\end{equation}
For \(u_0\in\L^2(\Omega)\) with
\(\norm[\L^1(\Omega)]{u_0}\leq\frac1{2C_{\text{GN}}}\), we deduce from \eqref{Ex2_7} that
\begin{equation}
\label{Ex2_8}
\norm[\L^2(\Omega)]{u_p(t)}^2+\norm[\L^2(0,t;\bfL^2(\Omega))]{\nabla u_p}^2
\leq
\norm[\L^2(\Omega)]{u_0}^2+ \int_0^t\norm[\L^2(\Omega)]{u_p(\tau)}^2\dd\tau.
\end{equation}
From \eqref{Ex2_8} and Gr\"onwall's lemma, we conclude that 
\begin{equation}
\label{fot1}
\{u_p\}_{p\in\En^*}
\textrm{ is bounded in } E_T\inj \L^{2}(0, T, \L^r(\Omega))
\end{equation}
for any \(1\leq r< \infty\).
Extracting
if necessary a subsequence, \eqref{fot1}
gives
\begin{equation}
\label{fot1501}
\nabla u_p 
\underset{p\rightarrow \infty}{\longrightarrow}
 \nabla u
 \text{ weakly in } \L^2(0, T, \bfL^2(\Omega)).
\end{equation}
Since \(\mathcal{V}_p(t) = 
\Delta_{D}^{-1}\big((Eu_p\star \phi_p)_{|_{\partial\Omega}}\big)\), \eqref{fot1}
and Lemma \ref{Lm4} leads to
\begin{equation}
\label{fot3}
\{\nabla\mathcal{V}_p\}_{p
\in\mathbb{N}^*} \textrm{ is  bounded
in }\L^{\infty}(0, T, \H^1(\Omega))
\inj \L^{\infty}(0, T, \L^r(\Omega))
\end{equation}
for any \(1\leq r< \infty\). 
Now, due to \eqref{PTQ},
\eqref{fot1}, and \eqref{fot3}
we also obtain that 
\(\bigl{\{}\frac{\dd u_p}{\dd t}\bigr{\}}_{p\in\En^*}\) 
is bounded in \(\L^2(0,T;\H^{-1}(\Omega))\).
By the Aubin-Lions lemma, extracting if necessary a subsequence, we conclude the 
existence of \(u\in\L^2(0,T;\L^2(\Omega))\) such that
\begin{equation*}
u_p \underset{p\rightarrow \infty}{\longrightarrow} u
\quad\text{in}\quad\L^2(\mathcal{Q}_T).
\end{equation*}
and we can moreover assume that \(u_p(t) \underset{p\rightarrow \infty}{\longrightarrow}
u(t)\) for a.e \(t\in(0, T)\) and   
\(\sup_{p\in\mathbb{N}^*}\norm[\L^2(\Omega)]{u_p}\in \L^2(0, T)\).  
It follows easily
that 
\begin{equation}
\label{fot4}
\nabla\mathcal{V}_p\underset{p\rightarrow \infty}{\longrightarrow}
 \nabla\Delta_{D}^{-1}u
\in \L^{2}(\mathcal{Q}_T).
\end{equation}
From 
\eqref{PTQ},
\eqref{fot1}, and \eqref{fot3} 
we also have that 
\(\bigl{\{}\frac{\dd u_p}{\dd t}\bigr{\}}_{p\in\En^*}\) 
is bounded in \(\L^2(0,T;(\bH^{1}(\Omega))')\),
hence, up to a subsequence, weakly--\(\star\) convergent 
in \(\L^2(0,T;(\bH^{1}(\Omega))')\).
With \eqref{fot1}, 
\eqref{fot1501}, \eqref{fot3},\eqref{fot4},
we may conclude that \(u\) satisfies the variational formulation in
\(\mathcal{P}_T\) with test functions
\(w\in \C^{\infty}([0, T]\times \bar{\Omega})\). By density, this holds
true for
\(w\in \L^{2}(0, T, \mathbf{H}^1({\Omega}))\).
The end of the proof is omitted.
\end{proof}

\renewcommand{\arraystretch}{0.91}{\small{ 
\paragraph*{Acknowledgments} }}

\end{document}